\documentclass[a4paper,oneside,11pt]{article}

\usepackage[a4paper]{geometry}
\usepackage{aeguill}                 
\usepackage{graphicx}
\usepackage{amsmath,amsfonts,amssymb,amsthm}
\usepackage{pstricks} 
\usepackage{epstopdf}
\usepackage{srcltx}
\usepackage[utf8]{inputenc} 
\usepackage[T1]{fontenc} 
\usepackage{hyperref} 
\usepackage[english]{babel} 
\usepackage{comment}

\title{Big mapping class groups with hyperbolic actions: classification and applications}
\author{Camille Horbez, Yulan Qing and Kasra Rafi}

\usepackage[all]{xy}
\usepackage{bbm}
\begin{document}
\maketitle
\newtheorem{de}{Definition} [section]
\newtheorem{theo}[de]{Theorem} 
\newtheorem{prop}[de]{Proposition}
\newtheorem{lemma}[de]{Lemma}
\newtheorem{cor}[de]{Corollary}
\newtheorem{propd}[de]{Proposition-Definition}
\newtheorem{conj}[de]{Conjecture}
\newtheorem{claim}{Claim}
\newtheorem*{claim2}{Claim}

\newtheorem{theointro}{Theorem}
\newtheorem*{defintro}{Definition}
\newtheorem{corintro}[theointro]{Corollary}

\theoremstyle{remark}
\newtheorem{rk}[de]{Remark}
\newtheorem{ex}[de]{Example}
\newtheorem{question}[de]{Question}

\normalsize

\newcommand{\Aut}{\mathrm{Aut}}
\newcommand{\Out}{\mathrm{Out}}
\newcommand{\Inn}{\mathrm{Inn}}
\newcommand{\stab}{\operatorname{Stab}}
\newcommand{\dunion}{\sqcup}
\newcommand{\eps}{\varepsilon}
\renewcommand{\epsilon}{\varepsilon}
\newcommand{\calf}{\mathcal{F}}
\newcommand{\cali}{\mathcal{I}}
\newcommand{\caly}{\mathcal{Y}}
\newcommand{\calx}{\mathcal{X}}
\newcommand{\calz}{\mathcal{Z}}
\newcommand{\calo}{\mathcal{O}}
\newcommand{\calb}{\mathcal{B}}
\newcommand{\calV}{\mathcal{V}}
\newcommand{\calq}{\mathcal{Q}}
\newcommand{\calu}{\mathcal{U}}
\newcommand{\call}{\mathcal{L}}
\newcommand{\bbR}{\mathbb{R}}
\newcommand{\bbZ}{\mathbb{Z}}
\newcommand{\bbD}{\mathbb{D}}
\newcommand{\NT}{\mathrm{NT}}
\newcommand{\actson}{\curvearrowright}
\newcommand{\caln}{\mathcal{N}}
\newcommand{\calg}{\mathcal{G}}
\newcommand{\calr}{\mathcal{R}}
\newcommand{\calt}{\mathcal{T}}
\newcommand{\calc}{\mathcal{C}}
\newcommand{\calm}{\mathcal{M}}
\newcommand{\adm}{\mathrm{adm}}
\newcommand{\cala}{\mathcal{A}}
\newcommand{\cals}{\mathcal{S}}
\newcommand{\calh}{\mathcal{H}}
\newcommand{\Stab}{\mathrm{Stab}}
\newcommand{\Isom}{\mathrm{Isom}}
\newcommand{\bdd}{\mathrm{bdd}}
\newcommand{\calp}{\mathcal{P}}
\newcommand{\Fix}{\mathrm{Fix}}
\newcommand{\cald}{\mathcal{D}}
\newcommand{\Mod}{\mathrm{Mod}}
\newcommand{\PML}{\mathrm{PML}}
\newcommand{\Map}{\mathrm{Map}}
\newcommand{\diam}{\mathrm{diam}}
\newcommand{\Disp}{\mathrm{Disp}}
\newcommand{\Bdd}{\mathrm{HB}}
\newcommand{\Ends}{\mathrm{Ends}}
\newcommand{\Homeo}{\mathrm{Homeo}}
\newcommand{\Axis}{\mathrm{Axis}}
\newcommand{\inj}{\mathrm{inj}}

\newcommand{\from}{{\colon \thinspace}}
\newcommand{\genus}{\mathrm{genus}}
\newcommand{\ens}{\mathrm{ens}}

\makeatletter
\edef\@tempa#1#2{\def#1{\mathaccent\string"\noexpand\accentclass@#2 }}
\@tempa\rond{017}
\makeatother
 
\newcommand{\Ccom}[1]{\Cmod\marginpar{\color{red}\tiny #1 --ch}} 
\newcommand{\Cmod}{$\textcolor{red}{\clubsuit}$} 

\newcommand{\Kcom}[1]{\Kmod\marginpar{\color{purple}\tiny #1 --ch}} 
\newcommand{\Kmod}{$\textcolor{purple}{\clubsuit}$} 

\newcommand{\Ycom}[1]{\Ymod\marginpar{\color{blue}\tiny #1 --ch}} 
\newcommand{\Ymod}{$\textcolor{blue}{\clubsuit}$}

\begin{abstract}
We address the question of determining which mapping class groups of infinite-type surfaces admit nonelementary continuous actions on hyperbolic spaces. 

More precisely, let $\Sigma$ be a connected, orientable surface of infinite type with tame endspace whose mapping class group is generated by a coarsely bounded subset. We prove that $\Map(\Sigma)$ admits a continuous nonelementary action on a hyperbolic space if and only if $\Sigma$ contains a finite-type subsurface which intersects all its homeomorphic translates. 

When $\Sigma$ contains such a nondisplaceable subsurface $K$ of finite type, the hyperbolic space we build is constructed from the curve graphs of $K$ and its homeomorphic translates via a construction of Bestvina, Bromberg and Fujiwara. Our construction has several applications: first, the second bounded cohomology of $\Map(\Sigma)$ contains an embedded $\ell^1$; second, using work of Dahmani, Guirardel and Osin, we deduce that $\Map(\Sigma)$ contains nontrivial normal free subgroups (while it does not if $\Sigma$ has no nondisplaceable subsurface of finite type), has uncountably many quotients and is SQ-universal. 
\end{abstract}

\section*{Introduction}

Let $\Sigma$ be a connected orientable surface, possibly of infinite type. We tackle the following two questions. Under what conditions on $\Sigma$ does the mapping class group $\Map(\Sigma)$ admit a continuous nonelementary isometric action on a hyperbolic space? When it does, what algebraic properties of $\Map(\Sigma)$ can one deduce from such an action?

\subsection*{Hyperbolic actions and nondisplaceable subsurfaces}

Concerning the first of these questions, the case where $\Sigma$ is a surface of finite type (i.e.\ with finitely generated fundamental group) has been famously answered by Masur and Minsky, who proved in \cite{MM} that the curve graph of any connected orientable surface of finite type is hyperbolic and admits a nonelementary action of $\Map(\Sigma)$, except in a few low-complexity cases. Their theorem is a milestone in the theory of mapping class groups of finite-type surfaces, thus motivating the analogous question for infinite-type surfaces.

There have been a lot of recent developments towards this question. In 2009, Calegari defined the ray graph of a plane minus a Cantor set \cite{Cal}, and conjectured that it is hyperbolic and unbounded, and that there are elements of $\Map(\Sigma)$ acting loxodromically on it. These conjectures have been answered positively by Bavard \cite{Bav}, which was the start of a lot of activity on big mapping class groups. Further developments, providing actions of $\Map(\Sigma)$ on hyperbolic graphs under various topological conditions on the surface, include \cite{AFP,DFV,FGM}, for instance. The study of the geometry of such graphs is still in constant expansion, see e.g.\ \cite{BDR,HHMV2,AGKMTW}.

From a different viewpoint, Mann and the third named author suggested in \cite{MR} that this question should be part of a more general framework, phrased in the language of Rosendal's approach to geometric group theory of (non-finitely generated, non-compactly generated) Polish groups \cite{Ros}. In particular, under a soft topological condition on the endspace of $\Sigma$ called \emph{tameness}\footnote{This is satisfied for instance if every end $\xi$ of $\Sigma$ has a neighborhood $U_\xi$ in the endspace which is stable in the sense that every subneighborhood contains a copy of $U_\xi$. See Section~\ref{sec:statement} for a complete definition.}, they classified which mapping class groups admit unbounded continuous isometric actions at all on metric spaces; they also classified which mapping class groups are \emph{CB-generated}, i.e.\ have a generating set that has bounded orbits in every continuous isometric action of $\Map(\Sigma)$. In addition, they coined the notion of a \emph{nondisplaceable} connected subsurface of $\Sigma$ -- defined as a subsurface $K$ such that $\phi(K)\cap K\neq\emptyset$ for every $\phi\in\Homeo(\Sigma)$ -- and established that the existence of nondisplaceable subsurfaces of finite type yields the existence of unbounded length functions on $\Map(\Sigma)$. The same notion was also independently introduced by Clay, Mangahas and Margalit \cite{CMM} in their work on normal right-angled Artin subgroups of mapping class groups of finite-type surfaces, where these were called \emph{orbit-overlapping} subsurfaces. The concept of nondisplaceable subsurfaces is key to our work, and we prove the following classification theorem -- in which we say that an isometric group action on a hyperbolic space is \emph{nonelementary} if it contains two independent loxodromic elements.

\begin{theointro}\label{theointro:hyperbolic-actions}
Let $\Sigma$ be a connected orientable surface of infinite type. 
\begin{enumerate}
\item If $\Sigma$ contains a nondisplaceable subsurface $K$ of finite type, then $\Map(\Sigma)$ acts continuously, nonelementarily by isometries on a hyperbolic space $\mathbb{X}$. In addition, the action can be chosen such that every element of $\Map(\Sigma)$ which preserves the isotopy class of $K$ and restricts to a pseudo-Anosov mapping class of $K$ has the WWPD property with respect to the $\Map(\Sigma)$-action on $\mathbb{X}$. 
\item Assume in addition that $\Map(\Sigma)$ is CB-generated and that the endspace of $\Sigma$ is tame. If $\Sigma$ does not contain any nondisplaceable subsurface of finite type, then $\Map(\Sigma)$ does not admit any continuous nonelementary isometric action on a hyperbolic space. 
\end{enumerate}
\end{theointro}

The WWPD property that arises in the first statement of the theorem, introduced by Bestvina, Bromberg and Fujiwara in \cite{BBF0}, gives precisions on the dynamics of the action. A possible definition is as follows. Let $G$ be a group acting by isometries on a hyperbolic space $\mathbb{X}$, and let $g\in G$ be a loxodromic element for the action. Denote by $(g^{-\infty},g^{+\infty})$ the pair of fixed points of $g$ in the Gromov boundary of $\mathbb{X}$. One says that $g$ has the \emph{WWPD property} if the $G$-orbit of $(g^{-\infty},g^{+\infty})$ is discrete in $\partial_\infty\mathbb{X}\times\partial_\infty\mathbb{X}\setminus\Delta$, where $\Delta$ denotes the diagonal. The WWPD property is a weakening of Bestvina and Fujiwara's WPD property \cite{BF} that was proved to hold for the action of a finite-type mapping class group on its curve graph -- but one cannot hope to have WPD elements in the infinite-type case in view of work of Bavard and Genevois \cite{BG}. Having WWPD elements yields applications to bounded cohomology, as explained in Corollary~\ref{corintro:cohomology} below and the paragraph that follows. 

The space $\mathbb{X}$ constructed in the proof of the first part of the theorem is a quasi-tree of metric spaces in the sense of Bestvina, Bromberg and Fujiwara \cite{BBF}, where the pieces are the curve graphs of all subsurfaces of $\Sigma$ in the $\Map(\Sigma)$-orbit of $K$. In the case of finite-type subsurfaces, the same construction appears in the recent work of Clay, Mangahas and Margalit \cite{CMM}. The fact that $K$ is nondisplaceable is exactly the assumption one needs to define projection maps (as in the work of Masur and Minsky \cite{MM2}) between these various curve graphs. The axioms needed to apply the Bestvina--Bromberg--Fujiwara machinery (and specifically, the Behrstock inequality and the axiom on finiteness of large projections) are checked as in the case of finite-type surfaces. The important point that allows to extend the construction to infinite-type surfaces is that the constants that arise in the projection axioms are uniform over all finite-type surfaces. 

\begin{figure}
\centering
\input{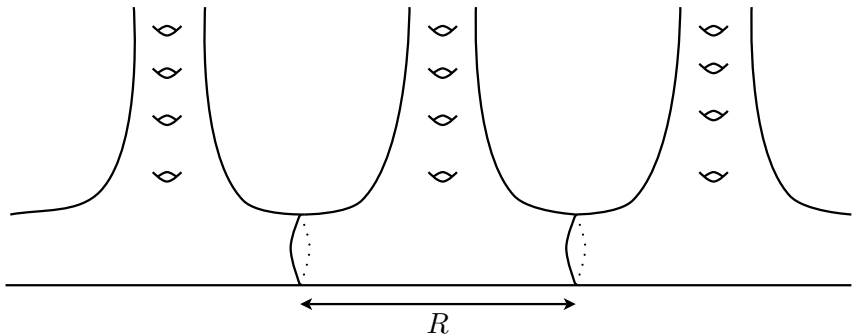}
\caption{An avenue of chimneys.}
\label{fig:chimneys}
\end{figure}

Let us now illustrate our proof of the second part of Theorem~\ref{theointro:hyperbolic-actions} on a concrete example, namely our surface $\Sigma$ is the \emph{avenue of chimneys} illustrated in Figure~\ref{fig:chimneys}. Assume that $\Map(\Sigma)$ acts continuously on a hyperbolic space $X$. The surface $\Sigma$ has two special ends, namely the leftmost and rightmost ends in the picture. The existence of these special ends is ensured in general by the assumptions made on $\Sigma$ together with work of Mann and the third named author \cite{MR}. Let $\Map^0(\Sigma)\subseteq\Map(\Sigma)$ be the finite-index subgroup made of mapping classes that preserve each of these two special ends, as opposed to permuting them. Then there is a shift homomorphism $\Map^0(\Sigma)\to\mathbb{Z}$, measuring the average displacement of chimneys to the right -- in general, one needs to consider finitely many shift homomorphisms. The kernel of this homomorphism to $\mathbb{Z}$ contains all mapping classes which are \emph{horizontally bounded}, i.e.\ supported on a subsurface that avoids a neighborhood of each of the two special ends (like $R$ in the picture). In fact all mapping classes in the kernel of the shift homomorphism are limits of mapping classes which are horizontally bounded. But every horizontally bounded mapping class can be displaced towards infinity, and by \cite{MR} this implies that it acts elliptically on $X$ -- in fact it acts with bounded orbits in any continuous isometric action of $\Map(\Sigma)$. Now we have a homomorphism to $\mathbb{Z}$ whose kernel is contained in the closure of a normal subgroup, all of whose elements act elliptically on $X$; it is then an exercise in actions on hyperbolic spaces to conclude that the $\Map(\Sigma)$-action on $X$ is elementary.

\subsection*{Applications to largeness properties}

We now discuss applications of our construction to algebraic properties of $\Map(\Sigma)$ in the presence of nondisplaceable subsurfaces of finite type. In fact we derive largeness properties for subgroups of $\Map(\Sigma)$ that contain sufficiently many elements acting as a pseudo-Anosov homeomorphism on a nondisplaceable subsurface of finite type.

Let $\Sigma$ be a connected, orientable surface, let $K\subseteq\Sigma$ be a subsurface of finite type, and let $G\subseteq\Map(\Sigma)$. We denote by $\Stab_G(K)$ the subgroup of $G$ made of all mapping classes that preserve the homotopy class of $K$. Denoting by $\widehat{K}$ a surface obtained from $K$ by gluing a once-punctured disk on every boundary component of $K$, every element of $\Stab_G(K)$ induces a mapping class of $\widehat{K}$. We say that a subgroup $G\subseteq\Map(\Sigma)$ is \emph{$K$-nonelementary} if $\Stab_{G}(K)$ contains two elements that induce independent pseudo-Anosov mapping classes of $\widehat{K}$. Our first application is to bounded cohomology.

\begin{corintro}\label{corintro:cohomology}
Let $\Sigma$ be a connected orientable surface, and let $K\subseteq\Sigma$ be a nondisplaceable subsurface. Let $G\subseteq\Map(\Sigma)$ be a $K$-nonelementary subgroup. 

Then the second bounded cohomology $H^2_b(G,\mathbb{R})$ contains an embedded copy of $\ell^1$.
\end{corintro}

This follows from a theorem of Handel and Mosher \cite[Theorem~2.10]{HM}, as Theorem~\ref{theointro:hyperbolic-actions} ensures that $G$ has a nonelementary action on a hyperbolic graph with at least one WWPD element. Corollary~\ref{corintro:cohomology} extends a theorem of Bestvina and Fujiwara for finite-type surfaces \cite{BF}, as well as earlier results of Bavard \cite[Théorème~4.8]{Bav} (answering a conjecture of Calegari \cite{Cal}), Bavard and Walker \cite[Theorem~9.1.1]{BW} and Rasmussen \cite[Corollary~1.2]{Ras} in the infinite-type setting. 

Our next application provides further largenss properties of $\Map(\Sigma)$. We state it here for the group $\Map(\Sigma)$ itself, but in fact the same statement holds true for many interesting subgroups of $\Map(\Sigma)$, and we refer to Theorem~\ref{theo:sq} for the full statement. We mention that the first conclusion of Theorem~\ref{theointro:largeness} partially answers a question raised by McLeay in \cite{McL}, by showing that the existence of a nondisplaceable subsurface of finite type in $\Sigma$ implies the existence of nongeometric normal subgroups of $\Map(\Sigma)$, i.e.\ normal subgroups whose automorphism group is not equal to the extended mapping class group of $\Sigma$.

\begin{theointro}\label{theointro:largeness}
Let $\Sigma$ be a connected orientable surface of infinite type which contains a nondisplaceable subsurface of finite type. Then
\begin{enumerate}
\item $\Map(\Sigma)$ contains a normal nonabelian free subgroup;
\item $\Map(\Sigma)$ contains uncountably many normal subgroups;
\item every countable group embeds in a quotient of $\Map(\Sigma)$.
\end{enumerate}
\end{theointro}

These statements were proved by Dahmani, Guirardel and Osin for finite-type surfaces \cite{DGO}; their techniques were recently developed by Clay, Mangahas and Margalit \cite{CMM} to find normal non-free right-angled Artin subgroups in the mapping class group. It is interesting that, in the infinite-type case, we still manage to get similar conclusions, despite $\Map(\Sigma)$ not being acylindrically hyperbolic \cite{BG}. We mention that the uncountably many normal subgroups we produce are all countable, made of mapping classes supported on finite-type subsurfaces. 

Our proof of Theorem~\ref{theointro:largeness} has similarities with the aforementioned work of Clay, Mangahas and Margalit -- in fact the first statement can probably be deduced from \cite[Theorem~1.6]{CMM}. It relies on the geometric small cancellation tools that were developed by Dahmani, Guirardel and Osin in \cite{DGO}, applied twice: once within the curve graph of a finite-type nondisplaceable subsurface, and once within the quasi-tree of metric spaces constructed in the proof of Theorem~\ref{theointro:hyperbolic-actions}. 

It is worth pointing out that allowing for quotients in the last conclusion of Theorem~\ref{theointro:largeness} is crucial, and it is not true in general that every countable group embeds in $\Map(\Sigma)$. In the case where $\Sigma$ is the plane minor a Cantor set, Calegari and Chen proved in \cite[Theorem~5.1]{CC} that a countable group $\Gamma$ embeds in $\Map(\Sigma)$ if and only if $\Gamma$ is circularly orderable. 

We also observe that the existence of nontrivial normal free subgroups, provided by the first conclusion of Theorem~\ref{theointro:largeness}, is in fact a characterization of the existence of a nondisplaceable subsurface of finite type.

\begin{theointro}\label{theointro:free-normal}
Let $\Sigma$ be a connected orientable surface of infinite type. Then $\Map(\Sigma)$ contains a nontrivial normal free subgroup if and only if $\Sigma$ contains a nondisplaceable subsurface of finite type.
\end{theointro}

The only if statement is proved as follows. Assume that $\Sigma$ contains no nondisplaceable subsurface of finite type. Given a normal subgroup $N\unlhd\Map(\Sigma)$, we show that we can always find a commutator of the form $k=(ghg^{-1})h^{-1}=g(hg^{-1}h^{-1})$ with $g\in N$ and $h$ finitely supported which is nontrivial. Such an element $k$ belongs to $N$, and is finitely supported, being the product of the finitely supported elements $ghg^{-1}$ and $h$. As we are assuming that $\Sigma$ has no nondisplaceable subsurface of finite type, the support of $k$ is displaceable off itself by a mapping class $\eta$, and $k$ and $\eta k\eta^{-1}$ generate a noncyclic abelian subgroup of $N$ -- in particular $N$ is not free. Details are given in Section~\ref{sec:normal-free-subgroups}.

\paragraph*{Organization of the paper.} 
Section~\ref{sec:prelims} collects background material regarding surfaces of infinite type and group actions on hyperbolic spaces. In Section~\ref{sec:nondisplaceable}, we study the case of surfaces that have a nondisplaceable subsurface of finite type and establish the first half of Theorem~\ref{theointro:hyperbolic-actions}; we also give our application to bounded cohomology of subgroups of $\Map(\Sigma)$. The proofs of Theorems~\ref{theointro:largeness} and~\ref{theointro:free-normal}, which study normal subgroups of $\Map(\Sigma)$ are given in Section~\ref{sec:largeness}. Finally, Section~\ref{sec:displaceable} is concerned with surfaces having no nondisplaceable subsurface of finite type: we prove the second half of Theorem~\ref{theointro:hyperbolic-actions}.  

\paragraph*{Acknowledgments.} 
We are very grateful to Dan Margalit for pointing us the connections between the present work and his work with Clay and Mangahas on normal right-angled Artin subgroups of mapping class groups.

The first named author acknowledges support from the Agence Nationale de la Recherche under Grant ANR-16-CE40-0006 DAGGER. The third named author was partially supported by the Discovery grants from the Natural Sciences and Engineering Research Council of Canada (RGPIN 06486).

\setcounter{tocdepth}{2}
\tableofcontents

\section{General background}\label{sec:prelims}

\subsection{Surfaces}

A \emph{surface} is a (boundaryless) $2$-dimensional topological manifold, i.e.\ a second-countable Hausdorff space $\Sigma$ such that every point in $\Sigma$ has an open neighborhood homeomorphic to an open subset of $\mathbb{R}^2$. The \emph{mapping class group} of a connected, orientable surface $\Sigma$ is defined as the group $\Map(\Sigma)$ of all isotopy classes of orientation-preserving homeomorphisms of $\Sigma$. The group $\Map(\Sigma)$ is equipped with the quotient topology of the compact-open topology on the group $\Homeo^+(\Sigma)$ of all orientation-preserving homeomorphisms of $\Sigma$.

Given a connected, orientable surface $\Sigma$, we let $g(\Sigma)$ be the genus of $\Sigma$ (possibly infinite), we let $E(\Sigma)$ be the end space of $\Sigma$ and $E^g(\Sigma)\subseteq E(\Sigma)$ be the subspace made of ends that are accumulated by genus (\emph{non-planar} in the terminology from \cite{Ric}). By a theorem of Richards \cite{Ric}, connected, orientable surfaces $\Sigma$ are classified up to homeomorphism by the triple $(g(\Sigma),E(\Sigma),E^g(\Sigma))$. 

A surface $\Sigma$ is \emph{of finite type} if it has finitely many connected components and the fundamental group of every such connected component is finitely generated. When $\Sigma$ is connected, we defined the \emph{complexity} of $\Sigma$ as $\xi(\Sigma)=3g(\Sigma)+|E(\Sigma)|-3$. Thus $\Sigma$ is of finite type if and only if $\xi(\Sigma)<+\infty$.

A \emph{bordered surface} is a topological space obtained from a surface $\Sigma$ by removing finitely many pairwise disjoint open disks. It is \emph{of finite type} if the surface $\Sigma$ can be chosen to be of finite type.

A \emph{subsurface} of a surface $\Sigma$ is a closed subset of $\Sigma$ whose boundary consists in a finite number of pairwise nonintersecting simple closed curves, such that none of these boundary curves bounds a disk or encloses an end of $\Sigma$. Every subsurface of $\Sigma$ is naturally a bordered surface; we say that a subsurface of $\Sigma$ is \emph{of finite type} if so is the corresponding bordered surface.

Given a subsurface $R\subseteq\Sigma$, the endspace of $R$ naturally embeds into the endspace of $\Sigma$; we let $\Ends(R)\subseteq E(\Sigma)$ be the image of this embedding.

As a matter of fact, every surface $\Sigma$ can be exhausted by an increasing sequence of subsurfaces of finite type, which in addition can be chosen to be connected. In particular every compact subset of $\Sigma$ is contained in a subsurface of $\Sigma$ of finite type.

Every subsurface $K\subseteq\Sigma$ of finite type determines a partition $\Pi_K$ of the ends of $\Sigma$. Given two subsets $X,Y\subseteq E(\Sigma)$, we say that $K$ \emph{separates} $X$ and $Y$ if $X$ and $Y$ belong to distinct subsets of the partition $\Pi_K$. 

\begin{lemma}\label{lemma:partition-ends}
Let $\Sigma$ be a connected orientable surface, and let $E$ be the end space of $\Sigma$. Let $X_1,\dots,X_k$ be finitely many pairwise disjoint closed subsets of $E$.

Then there exists a $k$-holed sphere $K\subseteq\Sigma$ that pairwise separates $X_1,\dots,X_k$.
\end{lemma}

\begin{proof}
We can find a clopen partition $E=Y_1\dunion\dots\dunion Y_k$ such that for every $i\in\{1,\dots,k\}$, one has $X_i\subseteq Y_i$. It is therefore enough to show that there exists a $k$-holed sphere in $\Sigma$ that pairwise separates $Y_1,\dots,Y_k$. For every $i\in\{1,\dots,k\}$, it follows from \cite{Ric} that there exists a bordered surface $\Sigma_i$ with a single boundary component whose endspace is homeomorphic to $Y_i$, so that the subspace made of ends accumulated by genus is homeomorphic to $Y_i\cap E^g$. In addition we can ensure that the sum of the genera of the surfaces $\Sigma_i$ is equal to the genus of $\Sigma$. Gluing a $k$-holed sphere along the boundary components of the surfaces $\Sigma_i$ thus yields a surface which is homeomorphic to $\Sigma$ by the classification of surfaces \cite{Ric}, and the lemma follows. 
\end{proof}

\subsection{Nondisplaceable subsurfaces}

A key concept in the work of Mann and the third named author \cite{MR}, which is also central in the present work, is that of a \emph{nondisplaceable} subsurface of $\Sigma$ -- most specifically, the important point is whether $\Sigma$ contains nondisplaceable subsurfaces of finite type. As in \cite[Definition~1.8]{MR}, we say that a connected subsurface $K\subseteq\Sigma$ is \emph{nondisplaceable} if for every $\phi\in\Homeo(\Sigma)$, one has $\phi(K)\cap K\neq\emptyset$. This definition can be extended to disconnected subsurfaces of $\Sigma$ in the following way (see \cite[Definition~2.7]{MR}): a subsurface $K\subseteq\Sigma$ is \emph{nondisplaceable} if for every $\phi\in\Homeo(\Sigma)$ and every connected component $K_1$ of $K$, there exists a connected component $K_2$ of $K$ such that \[\phi(K_1)\cap K_2\neq\emptyset.\] Notice that if $K\subseteq K'$ are two subsurfaces with $K$ nondisplaceable and $K'$ connected, then $K'$ is nondisplaceable. In particular, whenever a connected orientable surface $\Sigma$ contains a nondisplaceable subsurface of finite type, it actually contains one which is connected (since every connected surface has an exhaustion by connected subsurfaces of finite type).  

\subsection{Hyperbolic actions and the WWPD property}

We assume the reader to be familiar with basics on hyperbolic spaces in the sense of Gromov \cite{Gro} and isometric group actions on those. 

An isometric action of a group $G$ on a hyperbolic space $X$ is \emph{nonelementary} if $G$ has unbounded orbits in $X$ and does not have any finite orbit in the Gromov boundary $\partial_\infty X$ -- the terminology we use here departs from Gromov's \cite{Gro}, where this was called an action \emph{of general type}, but seems to prevail in the current literature. 

In the present paper, we will consider certain dynamical properties of isometric group actions on hyperbolic metric spaces. In particular, we will make use of the WWPD property, introduced by Bestvina, Bromberg and Fujiwara in \cite{BBF0} as a weakening of the WPD property introduced by Bestvina and Fujiwara in \cite{BF}. A possible definition is the following (see \cite[Proposition~2.3]{HM} for its equivalence with the original definition): given an isometric action of a group $G$ on a hyperbolic space $X$, an element $g\in G$ is \emph{WWPD} with respect to the $G$-action on $X$ if $g$ is loxodromic and the $G$-orbit of $(g^{-\infty},g^{+\infty})$ is a discrete subspace of $(\partial_\infty X\times\partial_\infty X)\setminus\Delta$ -- where $\Delta$ denotes the diagonal in $\partial_\infty X\times\partial_\infty X$. It is \emph{WPD} if in addition, the $G$-stabilizer of the pair $(g^{-\infty},g^{+\infty})$ is virtually cyclic (see \cite[Corollary~2.4]{HM}).

\section{Nondisplaceable subsurfaces and hyperbolic actions}\label{sec:nondisplaceable}

Let $\Sigma$ be an orientable surface. In this section, we prove that whenever $\Sigma$ contains a nondisplaceable subsurface $K$ of finite type, then $\Map(\Sigma)$ admits a nonelementary continuous action on a hyperbolic space. In addition this action can be constructed so that elements of $\Map(\Sigma)$ that restrict to pseudo-Anosov mapping classes on $K$ have the WWPD property with respect to the action. As a consequence, using a theorem of Handel and Mosher \cite{HM}, we deduce in Section~\ref{sec:cohomology} that the second bounded cohomology $H^2_b(\Map(\Sigma),\mathbb{R})$ is infinite-dimensional, and in fact contains an embedded copy of $\ell^1$.

\subsection{Curves, homotopies and mapping class groups}

\subsubsection{Curve graphs}\label{sec:curve-graph}

A simple closed curve on a surface $\Sigma$ is \emph{essential} if it does not bound a disk or a once-punctured disk. Let $\Sigma$ be a surface, and let $K\subseteq\Sigma$ be a subsurface of finite type. We let $\calc_{\Sigma}(K)$ be the graph whose vertices are the isotopy classes of simple closed curves on $\Sigma$ that have a representative contained in $K$ which is essential in $K$ (in particular not homotopic to one of the boundary curves of $K$), where two distinct isotopy classes are joined by an edge if they have disjoint representatives in $\Sigma$. As such the graph $\calc_{\Sigma}(K)$ is an induced subgraph of the curve graph $\calc(\Sigma)$ of $\Sigma$ -- i.e.\ the vertex set of $\calc_\Sigma(K)$ is a subset of $\calc(\Sigma)$, and two vertices are joined by an edge in $\calc_\Sigma(K)$ if and only if they are joined by an edge in $\calc(\Sigma)$. Viewing $K$ as a bordered surface, we can also consider its curve graph $\calc(K)$; the following two lemmas show that the inclusion $K\subseteq \Sigma$ induces an inclusion of $\calc(K)$ into $\calc(\Sigma)$ whose image is precisely $\calc_\Sigma(K)$. Also $\calc_\Sigma(K)$ only depends on the isotopy class of $K$: if $K$ and $K'$ are isotopic, then there is a natural identification between $\calc_\Sigma(K)$ and $\calc_\Sigma(K')$. 

\begin{lemma}\label{lemma:isotopy-curve}
Let $\Sigma$ be a surface, and let $K\subseteq\Sigma$ be a subsurface of finite type. Let $c$ and $c'$ be two essential simple closed curves on $\Sigma$ which are homotopic in $\Sigma$ and both contained and essential in $K$. Then $c$ and $c'$ are homotopic within $K$, i.e.\ there exists a homotopy $H:S^1\times [0,1]\to\Sigma$ with $H(S^1\times\{0\})=c$, $H(S^1\times\{1\})=c'$ and $H(S^1\times [0,1])\subseteq K$.
\end{lemma}

\begin{proof}
Let $H:S^1\times [0,1]\to \Sigma$ be a homotopy from $c$ to $c'$. Then $H(S^1\times [0,1])$ is compact, whence contained within a finite-type subsurface $K'\subseteq K$. The conclusion therefore follows from the particular case where the ambient surface is of finite type, established in \cite[Lemma~3.16]{FM}. 
\end{proof}

Similarly, the following lemma can be proved by reducing to the case of finite type surfaces.

\begin{lemma}\label{lemma:disjointness-curves}
Let $\Sigma$ be a surface, and let $K\subseteq\Sigma$ be a subsurface of finite type. Let $c$ and $c'$ be two essential simple closed curves on $\Sigma$ which are both contained in $K$. If the isotopy classes of $c$ and $c'$ have disjoint representatives, then they have disjoint representatives contained in $K$.
\qed
\end{lemma}

\subsubsection{Restriction homomorphisms}

Let $\Sigma$ be a surface, and let $K\subseteq\Sigma$ be an essential subsurface of finite type. Viewing $K$ as a bordered surface, we let $\widehat{K}$ be a surface obtained from $K$ by gluing a once-punctured disk on each boundary component of $K$. By \cite[Proposition~3.9]{FM}, the inclusion $K\hookrightarrow\widehat{K}$ induces a homomorphism $\Map(K)\to\Map(\widehat{K})$, whose kernel is free abelian, generated by twists about the boundary curves of $K$.

Let now $\Stab_{\Map(\Sigma)}(K)$ be the subgroup of $\Map(\Sigma)$ made of all mapping classes that preserve the isotopy class of $K$. The following lemma yields a homomorphism $\Stab_{\Map(\Sigma)}(K)\to\Map(\widehat{K})$.

\begin{lemma}
Let $\Sigma$ be a surface, and let $K\subseteq\Sigma$ be a subsurface of finite type. Then every $\Phi\in\Stab_{\Map(\Sigma)}(K)$ has a representative $\phi\in\Homeo(\Sigma)$ such that $\phi(K)=K$, and in addition any two such representatives induce the same element of $\Map(\widehat{\Sigma})$. 
\end{lemma}

\begin{proof}
Let $\phi_0$ be a representative of $\Phi$ in $\Homeo(\Sigma)$. Every isotopy $K\times [0,1]\to\Sigma$ between $K$ and $\phi_0(K)$ has image contained in a subsurface of $\Sigma$ of finite type, and therefore can be extended to an isotopy of $\Sigma$. The first part of the lemma follows. The additional part follows from the observation that any two such representatives have the same action on homotopy classes of essential simple closed curves of $K$.
\end{proof}

We also let $\Fix_{\Map(\Sigma)}(K)$ be the subgroup of $\Map(\Sigma)$ made of all elements that have a representative $\phi\in\Homeo(\Sigma)$ such that $\phi(K)=K$ and $\phi_{|K}=\mathrm{id}_K$.  We note that Dehn twists about boundary curves of $K$ belong to $\Fix_{\Map(\Sigma)}(K)$.

\begin{lemma}\label{lemma:ses-mcg}
Let $\Sigma$ be a surface, let $K\subseteq\Sigma$ be a subsurface of finite type, and let $\widehat{K}$ be a surface obtained from $K$ by gluing a once-punctured disk on each boundary component of $K$. 

Then there exists a homomorphism $\Stab_{\Map(\Sigma)}(K)\to\Map(\widehat{K})$ whose kernel is equal to $\Fix_{\Map(\Sigma)}(K)$.
\end{lemma}

\begin{proof}
Every element in the kernel of this homomorphism has a representative $\phi\in\Homeo(\Sigma)$ such that $\phi(K)=K$, and being in the kernel implies that this representative $\phi$ is a product of peripheral Dehn twists. One can thus isotope $\phi$ to get a representative $\phi'$ such that $\phi'_{|K}=\mathrm{id}$.
\end{proof}

\begin{lemma}\label{lemma:fix-stab-curves}
Let $\Sigma$ be a surface, and let $K\subseteq\Sigma$ be a subsurface of finite type. 
\begin{enumerate}
\item The setwise stabilizer of $\calc_\Sigma(K)$ in the $\Map(\Sigma)$-action on $\calc(\Sigma)$ is equal to $\Stab_{\Map(\Sigma)}(K)$.
\item The pointwise stabilizer of $\calc_\Sigma(K)$ in the $\Map(\Sigma)$-action on $\calc(\Sigma)$ is equal to $\Fix_{\Map(\Sigma)}(K)$.
\end{enumerate}
\end{lemma}

\begin{proof}
An element of $\Map(\Sigma)$ is in the stabilizer of $\calc(\Sigma)$ if and only if it has a representative $\phi$ that fixes the boundary curves of $K$ (and therefore such that $\phi(K)=K$). The second statement then follows from the case of finite-type surfaces.
\end{proof}

\subsection{Review on quasi-trees of metric spaces}\label{sec:bbf}

We now review the celebrated construction of Bestvina, Bromberg and Fujiwara from \cite{BBF} which will be used in the next section. 

An action of a group $G$ on a collection $\mathbb{Y}$ of metric spaces is \emph{metric-preserving} if for every $g\in G$ and every $Y\in\mathbb{Y}$, there exists an isometry $\iota_g^Y:Y\to gY$, so that for all $g,h\in G$ and every $Y\in\mathbb{Y}$, one has $\iota_{gh}^Y=\iota_g^{hY}\circ\iota_h^Y$. A \emph{$G$-equivariant projection family} is a pair $(\mathbb{Y},(\pi_Y(Z))_{Y\neq Z\in\mathbb{Y}})$ where 
\begin{itemize}
\item $\mathbb{Y}$ is a collection of metric spaces equipped with a metric-preserving $G$-action, 
\item  $\pi_Y(Z)$ is a nonempty subset of $Y$ for any two distinct $Y,Z\in\mathbb{Y}$,
\item for every $g\in G$ and any two distinct $Y,Z\in\mathbb{Y}$, one has $\pi_{gY}(gZ)=\iota_g^Y(\pi_Y(Z))$.
\end{itemize}  

\begin{de}\label{de:bbf}
Let $G$ be a group. A $G$-equivariant projection family $(\mathbb{Y},(\pi_Y(Z))_{Y\neq Z\in\mathbb{Y}})$ is a \emph{BBF family} for $G$ if, letting $d_Y(X,Z):=\diam(\pi_Y(X)\cup\pi_Y(Z))$ for every $X,Y,Z\in\mathbb{Y}$ with $Y\neq X,Z$, there exists $\theta>0$ such that the following conditions hold:
\begin{itemize}
	\item[(P0)] For all distinct $X,Y\in\mathbb{Y}$, one has $d_Y(X,X) \leq \theta$;
	\item[(P1)] For all pairwise distinct $X,Y,Z\in\mathbb{Y}$, if $d_Y(X,Z)> \theta$ then $d_X(Y,Z) \leq \theta$;
	\item[(P2)] For all $X,Z\in\mathbb{Y}$, the set $\{Y \neq X,Z \mid d_Y(X,Z) > \theta\}$ is finite.
\end{itemize}
\end{de}

The following statement records the output of the Bestvina--Bromberg--Fujiwara construction.

\begin{theo}[{Bestvina--Bromberg--Fujiwara \cite{BBF}}]\label{theo:bbf}
Let $G$ be a group. Assume that there exists a BBF family $(\mathbb{Y},(\pi_Y(Z))_{Y\neq Z\in\mathbb{Y}})$ for $G$, with all spaces in $\mathbb{Y}$ uniformly hyperbolic. Then $G$ acts by isometries on a hyperbolic metric space $\calc(\mathbb{Y})$ with the following properties:
\begin{enumerate}
\item every $Y\in\mathbb{Y}$ embeds as a geodesically convex subspace of $\calc(\mathbb{Y})$,
\item there exists $D>0$ such that for all $Y,Z\in\mathbb{Y}$ with $Y\neq Z$, one has $\Delta_{\calc(\mathbb{Y})}(Y,Z)<D$, 
\item for every $Y\in\mathbb{Y}$ and every $g\in\Stab_G(Y)$, if $g$ is loxodromic WPD for the action of $\Stab_G(Y)/\Fix_G(Y)$ on $Y$, then $g$ is loxodromic WWPD for the $G$-action on $\calc(\mathbb{Y})$.
\end{enumerate} 
\end{theo}

\begin{proof}
Let $\calc(\mathbb{Y})$ be the quasi-tree of metric spaces defined in \cite[Definition~4.1]{BBF}. As all spaces in $\mathbb{Y}$ are uniformly hyperbolic, the space $\calc(\mathbb{Y})$ is hyperbolic \cite[Theorem~4.17]{BBF}. Every space in $\mathbb{Y}$ is geodesically convex in $\calc(\mathbb{Y})$ by \cite[Lemma~4.2]{BBF}. The existence of the constant $D$ follows from the fact that projections have uniformy bounded diameter in view of Axiom~$(P0)$ together with \cite[Corollary~4.10]{BBF}. The conclusion about WWPD isometries is given in \cite[Proposition~4.20]{BBF}. 
\end{proof}

\subsection{Quasi-trees of curve graphs for big mapping class groups}

Let $\Sigma$ be a connected orientable surface. Assume that $\Sigma$ contains a connected nondisplaceable subsurface $K$ of finite type, and denote by $[K]$ the isotopy class of $K$. Then $\Map(\Sigma)$ acts in a metric-preserving way on $$\mathbb{Y}_K:=\{\calc_\Sigma([K'])|[K']\in\Map(\Sigma)\cdot [K]\}.$$ As $K$ is nondisplaceable and of finite type, given any two distinct subsurfaces $K_1,K_2\in\Homeo(\Sigma)\cdot K$, at least one of the boundary components of $K_2$ intersects the subsurface $K_1$ in an essential curve or arc. This observation yields, for every $K_1\in\Map(\Sigma)\cdot K$, a projection $$\pi_{\calc_\Sigma(K_1)}(\calc_{\Sigma}(K_2))\subseteq\calc_\Sigma(K_1),$$ equal to the set of all isotopy classes of essential simple closed curves on $K_1$ that are disjoint from some boundary curve of $K_2$ that intersects $K_1$, see \cite{MM2}. The family $\mathbb{Y}_K$ together with these projections form a $\Map(\Sigma)$-equivariant projection family.

\begin{prop}\label{prop:bbf-family}
Let $\Sigma$ be a connected orientable surface which contains a connected nondisplaceable subsurface $K$ of finite type.

Then the projection family $(\mathbb{Y}_K,(\pi_Y(Z))_{Y\neq Z\in\mathbb{Y}_K})$ is a BBF family for $\Map(\Sigma)$.
\end{prop}  

\begin{proof}
The argument is the same as in the proof of \cite[Proposition~3.2]{CMM}. Given $Y\in\mathbb{Y}$ and $X,Z\in\mathbb{Y}\setminus\{Y\}$, we let $d_Y(X,Z):=\diam_Y(\pi_Y(X)\cup\pi_Y(Z))$. We will now check the various projection axioms from Definition~\ref{de:bbf}.

Condition (P0) is satisfied because the projections have uniformly bounded diameter.

We now check Condition (P1). Let $X,Y,Z\in\mathbb{Y}$ be pairwise distinct -- these are curve complexes associated to finite-type subsurfaces $K_X,K_Y,K_Z$ of $\Sigma$. As $K_X\cup K_Y\cup K_Z$ is compact, it is contained in some subsurface $\widetilde{K}\subseteq \Sigma$ of finite type. By working within the surface $\widetilde{K}$, Condition~(P1) follows from \cite[Lemma~5.2]{BBF} -- notice indeed that the threshold constant $\theta$ given that lemma is independent from the topology of $\widetilde{K}$.

We finally check Condition (P2). The proof comes from \cite[Lemma~5.3]{BBF}. It is enough to prove that given any two essential simple closed curves $x,y$ on $\Sigma$, there are only finitely many isotopy classes of subsurfaces $K'$ in the $\Map(\Sigma)$-orbit of $K$ such that $x$ and $y$ both intersect $K'$, and whenever $x',y'$ are simple closed curves that are contained and essential in $K'$, and are disjoint from $x$ and $y$, respectively, then $d_{\calc_\Sigma(K')}(x',y')>10$. Let $K_{xy}\subseteq \Sigma$ be the smallest subsurface of $\Sigma$ (of finite type) that contains $x$ and $y$: this is well-defined up to isotopy. If $K'$ cannot be isotoped to be contained in $K_{xy}$, then there is a curve or an arc in $K'\setminus K_{xy}$, so $d_{\calc(K')}(x',y')$ is bounded. We can thus restrict to only considering subsurfaces $K'$ with $K'\subseteq K_{xy}$, and in this case then the result follows from the finite type case \cite[Lemma~5.3]{BBF} -- again it is important to observe that the constants given by that lemma are independent from the topology of $K'$.
\end{proof}

Given a subsurface $K$ of $\Sigma$, we say that an element $f\in\Map(\Sigma)$ is \emph{$K$-pseudo-Anosov} if $f$ preserves the isotopy class of $K$ and, denoting by $\widehat{K}$ a surface obtained from $K$ by gluing a once-punctured disk on every boundary component of $K$, the mapping class $f$ induces a pseudo-Anosov mapping class of $\widehat{K}$.

\begin{theo}\label{theo:big-bbf}
Let $\Sigma$ be a connected orientable surface with $\xi(\Sigma)>0$, and assume that $\Sigma$ contains a nondisplaceable connected subsurface $K$ of finite type.

Then there exists an unbounded hyperbolic space $\mathbb{X}$ equipped with a continuous nonelementary isometric action of $\Map(\Sigma)$ such that every element of $\Map(\Sigma)$ which is $K$-pseudo-Anosov is a WWPD loxodromic element for the $\Map(\Sigma)$-action on $\mathbb{X}$.
\end{theo}

\begin{proof}
Let $K\subseteq \Sigma$ be a nondisplaceable connected subsurface of $\Sigma$ of finite type; without loss of generality we can assume that $K$ supports a pseudo-Anosov mapping class. All graphs $\calc_{\Sigma}([K'])$ with $[K']\in\Map(\Sigma)\cdot [K]$ are isomorphic to the curve graph of $K$; in particular, in view of work of Masur and Minsky \cite{MM}, they are all uniformly hyperbolic and unbounded. Theorem~\ref{theo:bbf} thus yields us an unbounded hyperbolic space $\mathbb{X}=\calc(\mathbb{Y}_K)$ associated to the BBF family $(\mathbb{Y}_K,(\pi_Y(Z))_{Y\neq Z\in\mathbb{Y}_K})$, on which $\Map(\Sigma)$ acts by isometries.

The action of $\Map(\Sigma)$ on $\mathbb{X}$ is continuous because if $f\in\Map(\Sigma)$ and $(f_n)_{n\in\mathbb{N}}\in\Map(\Sigma)^\mathbb{N}$ converges to $f$, then for every isotopy class $K$ of finite-type subsurfaces, the images $f_n(K)$ are eventually constant, and for every isotopy class $c$ of essential simple closed curves, the images $f_n(c)$ are eventually constant. 

We finally check that all $K$-pseudo-Anosov elements of $\Map(\Sigma)$ are loxodromic WWPD elements for the $\Map(\Sigma)$-action on $\mathbb{X}$. By work of Bestvina and Fujiwara \cite[Proposition~11]{BF}, the action of every pseudo-Anosov element of $\Map(\widehat{K})$ on the curve graph $\calc(\widehat{K})$ is WPD. The conclusion therefore follows from Theorem~\ref{theo:bbf}, together with the identification of $\Stab_{\Map(\Sigma)}(\calc_\Sigma(K))/\Fix_{\Map(\Sigma)}(\calc_\Sigma(K))$ with a subgroup of $\Map(\widehat{K})$, provided by Lemmas~\ref{lemma:ses-mcg} and~\ref{lemma:fix-stab-curves}.    
\end{proof}

\subsection{Application to bounded cohomology}\label{sec:cohomology}

Given a subgroup $G\subseteq\Map(\Sigma)$ and a subsurface $K\subseteq\Sigma$ of finite type, we let $\Stab_G(K)$ be the subgroup of $G$ made of all mapping classes that preserve the isotopy class of $K$ -- in other words $\Stab_G(K)=\Stab_{\Map(\Sigma)}(K)\cap G$.

\begin{de}\label{de:K-nonelementary}
Let $\Sigma$ be a connected orientable surface, and let $K\subseteq\Sigma$ be a finite-type subsurface. Let $\widehat{K}$ be a surface obtained from $K$ by gluing a once-punctured disk on every boundary component of $K$. 

A subgroup $G\subseteq\Map(\Sigma)$ is \emph{$K$-nonelementary} if the image of $\Stab_{G}(K)$ in $\Map(\widehat{K})$ contains two independent pseudo-Anosov mapping classes of $\widehat{K}$.
\end{de} 

A theorem of Handel and Mosher \cite[Theorem~2.10]{HM} asserts that for every group $G$ acting isometrically on a hyperbolic space with two independent loxodromic elements and at least one loxodromic WWPD element, the second bounded cohomology $H^2_b(G,\mathbb{R})$ contains an embedded $\ell^1$. As a consequence of Theorem~\ref{theo:big-bbf}, we thus obtain the following corollary, which generalizes earlier results of Bavard and Walker \cite[Theorem~9.1.1]{BW} and Rasmussen \cite[Corollary~1.2]{Ras}. Again, the case of finite-type surfaces is due to Bestvina and Fujiwara \cite{BF}.

\begin{cor}
Let $\Sigma$ be a connected orientable surface with $\xi(\Sigma)>0$, and assume that $\Sigma$ contains a nondisplaceable subsurface of finite type. Then $H^2_b(\Map(\Sigma),\mathbb{R})$ contains an embedded $\ell^1$.

More generally, let $K$ be a nondisplaceable connected subsurface of finite type of $\Sigma$, and let $H\subseteq\Map(\Sigma)$ be a subgroup which is $K$-nonelementary. Then $H^2_b(H,\mathbb{R})$ contains an embedded $\ell^1$. 
\qed
\end{cor}

\section{Normal subgroups via geometric small cancellation}\label{sec:largeness}

The goal of the present section is to prove Theorem~\ref{theointro:largeness} from the introduction: if $\Sigma$ contains a nondisplaceable subsurface of finite type, then $\Map(\Sigma)$ contains uncountably many normal subgroups, including nonabelian free normal subgroups. This is proved in Section~\ref{sec:sq} below, after we recalled background on geometric small cancellation as developed by Dahmani, Guirardel and Osin \cite{DGO} in Section~\ref{sec:dgo}. Finally, in Section~\ref{sec:normal-free-subgroups}, we show that in fact $\Map(\Sigma)$ contains nontrivial normal free subgroups exactly when $\Sigma$ has nondisplaceable subsurfaces of finite type.

\subsection{Review on geometric small cancellation}\label{sec:dgo}

We now review work of Dahmani, Guirardel and Osin \cite{DGO} that will be used in the next section, following the exposition from \cite{Gui} and \cite{Cou}. 

Let $X$ be a hyperbolic metric space equipped with an isometric action of a group $G$, and let $\delta\ge 0$ be the hyperbolicity constant of $X$. Given a subset $Q\subseteq X$ and $r\ge 0$, we denote by $Q^{+r}$ the closed $r$-neighborhood of $Q$. Given two subsets $Q_1,Q_2\subseteq X$, we define the \emph{overlap} of $Q_1$ and $Q_2$ as $$\Delta_X(Q_1,Q_2):=\diam(Q_1^{+20\delta}\cap Q_2^{+20\delta}).$$

We say that $Q$ is \emph{almost convex} if given any two points $x,y\in Q$, there exist $x',y'\in Q$ with $d(x,x')\le 8\delta$ and $d(y,y')\le 8\delta$, such that there exist geodesic segments $[x,x'],[x',y'],[y',y]$ that are contained in $Q$.

A \emph{moving pair} for the $G$-action on $X$ is a pair $(H,Q)$, where $Q$ is an almost convex subset of $X$ and $H$ is a normal subgroup of the setwise $G$-stabilizer of $Q$, which we denote by $\Stab_G(Q)$.
The \emph{injectivity radius} of the pair $(H,Q)$ is defined as $$\inj_X(H,Q)=\inf\{d_X(x,hx)|x\in Q, h\in H \setminus\{1\}\},$$ and its \emph{fellow traveling constant} is $$\Delta^*_X(H,Q)=\sup\{\Delta(Q,tQ)|t\in G\setminus\Stab_G(Q)\}.$$ Given $A>0$ and $\epsilon>0$, we say that a moving pair $(H,Q)$ satisfies the \emph{$(A,\epsilon)$-small cancellation condition} if $\inj_X(H,Q)\ge A$ and $\Delta^*_X(H,Q)\le\epsilon\cdot\inj_X(H,Q)$.

\begin{figure}
\centering
\input{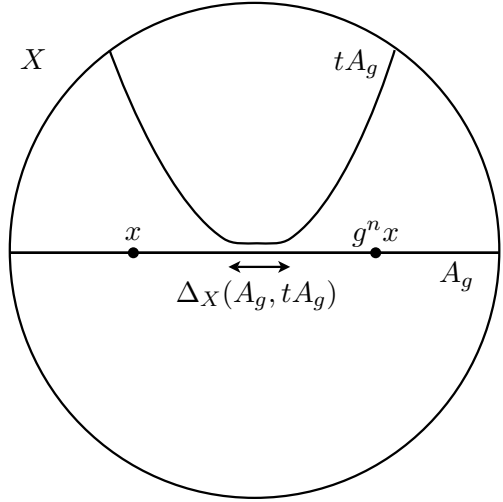}
\caption{The small cancellation condition.}
\label{fig:dgo}
\end{figure}

\begin{ex}
A first example to keep in mind is the following: let $G$ be a group acting on a hyperbolic space $X$, and let $g\in G$ be a WPD loxodromic element. For simplicity, let us assume that $G$ is torsion-free and that $g$ has an invariant axis $A_g$ on which it acts by translation -- for example $G$ could be a surface group acting properly discontinuously on the hyperbolic plane $\mathbb{H}^2$. Then for every $n\in\mathbb{N}$, the pair $(A_g,\langle g^n\rangle)$ is a moving pair. Using the WPD property, we can always choose $n\in\mathbb{N}$ large enough to ensure that the overlap between distinct $G$-translates of $A_g$ is always small compared to the translation length of $g^n$ (as in Figure~\ref{fig:dgo}). In other words, for every $A>0$ and $\epsilon>0$, there exists $n\in\mathbb{N}$ so that the moving pair $(A_g,\langle g^n\rangle)$ satisfies the $(A,\epsilon)$-small cancellation condition.
\end{ex}

We define the \emph{translation length} of an isometry $g$ of $X$ as $||g||_X=\inf_{x\in X}d(x,gx)$. The following theorem is (in a slightly simplified version) the main theorem of geometric small cancellation theory, due to Dahmani, Guirardel and Osin \cite{DGO} (see also \cite[Theorem~1.3]{Gui}).

\begin{theo}[Dahmani--Guirardel--Osin \cite{DGO}]\label{theo:dgo}
Let $G$ be a group, and let $X$ be a hyperbolic metric space equipped with an isometric $G$-action. For every $C\ge 0$, there exist $A=A(C)>0$ and $\epsilon=\epsilon(C)>0$ such that for every moving pair $(H,Q)$ satisfying the $(A,\epsilon)$-small cancellation condition, the following hold:
\begin{enumerate}
\item there exists a subset $Z\subseteq G$ such that the normal subgroup $\langle\langle H\rangle\rangle$ of $G$ generated by $H$ satisfies $\langle\langle H\rangle\rangle=\ast_{t\in Z} (tHt^{-1})$;
\item the projection $G\to G/\langle\langle H\rangle\rangle$ induces an injective homomorphism $\Stab_G(Q)/H\to G/\langle\langle H\rangle\rangle$; 
\item all nontrivial elements of $\langle\langle H\rangle\rangle$ act loxodromically on $X$ with translation length at least $C$.
\end{enumerate}
\end{theo}

Following \cite{DGO}, given a group $G$ and an element $g\in G$, we define the \emph{elementary subgroup} $\mathcal{E}_G(g)$ as $$\mathcal{E}_G(g)=\{h\in G|\exists n,m\in\mathbb{Z}\setminus\{0\}, g^m=hg^nh^{-1}\}.$$ We will need the following proposition, which provides a situation where one can apply the geometric small cancellation theorem. See \cite[Section~6.2]{DGO} or \cite[Proposition~2.15]{Gui}. 

\begin{prop}\label{prop:small-cancellation-free-subgroup}
Let $G$ be a group, and let $X$ be a hyperbolic metric space equipped with a nonelementary isometric $G$-action. Let $H\subseteq G$ be a subgroup. Assume that $H$ is not cyclic, and that there exists $g\in H$ which acts loxodromically on $X$, has the WPD property with respect to the $G$-action on $X$, and such that $\mathcal{E}_G(g)$ is isomorphic to $\mathbb{Z}$. 

Then for every $A>0$ and every $\epsilon>0$, there there exists a rank $2$ free subgroup $F\subseteq H$ and an almost convex subspace $Q_F\subseteq X$ with $\Stab_G(Q_F)=F$, such that $(F,Q)$ satisfies the $(A,\epsilon)$-small cancellation condition. 
\end{prop} 

\subsection{Abundance of normal subgroups}\label{sec:sq}

Given a group $G$ and an element $g\in G$, we say that $g$ is \emph{symmetryless} if $\mathcal{E}_G(g)=\langle g\rangle$. We recall the definition of a $K$-nonelementary subgroup from Definition~\ref{de:K-nonelementary}

\begin{de}
Let $\Sigma$ be a connected orientable surface, and let $K\subseteq\Sigma$ be a finite-type subsurface. Let $\widehat{K}$ be a surface obtained from $K$ by gluing a once-punctured disk on every boundary component of $K$. 

A subgroup $G\subseteq\Map(\Sigma)$ is \emph{$K$-generic} if $G$ is $K$-nonelementary, and in addition $G$ contains an element $g$ supported on $K$ whose image in $\Map(\widehat{K})$ is a symmetryless pseudo-Anosov mapping class. 
\end{de}

Notice that when $\Sigma$ is an infinite-type surface that contains a nondisplaceable subsurface of finite type, one can always find a nondisplaceable subsurface $K\subseteq\Sigma$ of finite type such that $\Map(\Sigma)$ is $K$-generic -- this can be viewed using \cite[Lemma~6.18]{DGO}, for instance. This condition is also satisfied by many interesting subgroups of $\Map(\Sigma)$, such as the Torelli subgroup, the pure mapping class group (acting trivially on the space of ends), or the countable subgroup made of finitely supported elements, for instance.

We would also like to observe that a $K$-generic subgroup $G$ of $\Map(\Sigma)$ actually contains many elements that are supported on $K$: indeed, starting from a pseudo-Anosov element $f$ supported on $K$ and an element $g$ that preserves $K$ and induces a pseudo-Anosov mapping class of $\widehat{K}$, one builds a new element supported on $K$ by considering the commutator $[f,g]$.

Recall that a group $G$ is \emph{SQ-universal} if every countable group embeds in some quotient of $G$. The goal of the present section is to prove the following theorem, which is an elaboration on Theorem~\ref{theointro:largeness} from the introduction.

\begin{theo}\label{theo:sq}
Let $\Sigma$ be a connected orientable surface with $\xi(\Sigma)>0$, and assume that $\Sigma$ contains a nondisplaceable subsurface $K$ of finite type. Let $G\subseteq\Map(\Sigma)$ be a $K$-generic subgroup. Then the following conclusions hold.
\begin{enumerate}
\item The group $G$ contains a normal nonabelian free subgroup.
\item There exists a rank $2$ free subgroup $F\subseteq G$ such that for every countable group $\Gamma$, there exists a quotient $\theta:G\twoheadrightarrow \overline{G}$ such that $\Gamma$ embeds in $\theta(F)$. In particular $G$ is SQ-universal and contains uncountably many normal subgroups.
\end{enumerate}
\end{theo}

Here is a comment about the second conclusion: the last part of this conclusion, namely, that $G$ contains uncountably many normal subgroups, follows from the first in the following way. Since $F$ is countable, any given quotient of $F$ is countable, and therefore can only contain countably many $2$-generated subgroups. On the other hand, there are uncountably many $2$-generated groups, and every such group embeds in a quotient of $F$ of the form $F/(F\cap N)$ for some normal subgroup $N\unlhd G$. Therefore, the groups $F\cap N$ take uncountably many values as $N$ ranges over all normal subgroups of $G$. This shows that $G$ has uncountably many normal subgroups.

\begin{rk}
Theorem~\ref{theo:sq} shows in particular that every countable group embeds in a quotient of $\Map(\Sigma)$. It is natural to ask whether there exist surfaces $\Sigma$ for which every countable group embdes in $\Map(\Sigma)$, without having to pass to quotients. Aougab, Patel and Vlamis \cite{APV2} have recently announced that such examples exist, but to our knowledge there is no example so far where $\Sigma$ contains nondisplaceable surfaces. In fact, this stronger statement is not true of all surfaces $\Sigma$: Calegari and Chen proved in \cite[Theorem~5.1]{CC} that a countable group embeds in the mapping class group of the plane minus a Cantor set if and only if it is circularly orderable. 

Notice also that Theorem~\ref{theo:sq} also states the SQ-universality of many countable subgroups of $\Map(\Sigma)$, for which this stronger conclusion can never hold as there are uncountably many $2$-generated groups.

We also mention that it is also natural to further ask the following question: which countable groups quasi-isometrically embed in $\Map(\Sigma)$ when this group is CB-generated?   
\end{rk}

\begin{rk}
We can be a bit more precise about the output of our construction. In particular, the group $F$ that arises in the statement of Theorem~\ref{theo:sq} is made of elements that are supported on $K$, and all elements in the normal subgroups $N$ we produce are countable, supported on finite-type subsurfaces.
\end{rk}

We now turn to the proof of Theorem~\ref{theo:sq}. We make the following definition.

\begin{de}\label{de:small-cancellation-subgroup}
Let $\Sigma$ be a connected orientable surface, let $K\subseteq\Sigma$ be a connected nondisplaceable subsurface of finite type, and let $G\subseteq\Map(\Sigma)$ be a subgroup. Let $H\subseteq \Stab_G(K)$ be a subgroup made of elements that are supported on $K$, and let $Q\subseteq\calc_\Sigma(K)$ be a subspace. Denote by $\widehat{G}$ and $\widehat{H}$ the respective images of $\Stab_G(K)$ and $H$ in $\Map(\widehat{K})$, and by $\widehat{Q}$ the image of $Q$ under the natural identification between $\calc_\Sigma(K)$ and $\calc(\widehat{K})$.

Let $A>0$ and $\epsilon>0$. We say that $(H,Q)$ is an \emph{$(A,\epsilon)$-small cancellation pair} if the following conditions hold:
\begin{enumerate}
\item $(\widehat{H},\widehat{Q})$ is a moving pair and satisfies the $(A,\epsilon)$-small cancellation condition with respect to the $\widehat{G}$-action on $\calc(\widehat{K})$, and 
\item the $\widehat{G}$-stabilizer of $\widehat{Q}$ has a lift in $\Stab_G(K)$ containing $H$ and made of elements that are supported on $K$.
\end{enumerate} 
We say that $H$ is an \emph{$(A,\epsilon)$-small cancellation subgroup} if there exists $Q\subseteq\calc_\Sigma(K)$ such that $(H,Q)$ is an $(A,\epsilon)$-small cancellation pair.
\end{de}

Our proof of Theorem~\ref{theo:sq} relies on three essential steps: finding appropriate small cancellation subgroups in $\Stab_G(K)$, applying small cancellation theory to the $\Stab_G(K)$-action on the curve graph $\calc_\Sigma(K)$, and applying small cancellation theory once more to the $G$-action on the quasi-tree of metric spaces constructed in Theorem~\ref{theo:big-bbf}.

\paragraph*{Step 1: Small cancellation on the curve graph of $\widehat{K}$.}

\begin{lemma}\label{lemma:small-cancellation-on-curve-graph}
For every $L>0$, there exist $A=A(L)>0$ and $\epsilon=\epsilon(L)>0$ such that the following holds. Let $\Sigma$ be a connected orientable surface which contains a connected nondisplaceable subsurface $K$ of finite type, and let $G\subseteq\Map(\Sigma)$ be a subgroup. Let $H\subseteq\Stab_G(K)$ be a subgroup made of elements supported on $K$, and let $\langle\langle H\rangle\rangle$ be the normal subgroup of $\Stab_G(K)$ generated by $H$. Assume that there exists a subset $Q\subseteq\calc_\Sigma(K)$ such that $(H,Q)$ is an $(A,\epsilon)$-small cancellation pair. 

Then $(H,Q)$ is a moving pair and satisfies the $(A,\epsilon)$-small cancellation condition with respect to the action of $\Stab_G(K)$ on $\calc_\Sigma(K)$. In particular,
\begin{enumerate}
\item $\langle\langle H\rangle\rangle$ is equal to a free product of conjugates of $H$;
\item denoting by $G_Q$ the setwise stabilizer of $Q$ in $\Stab_G(K)$, the group $H$ is normal in $G_Q$ and the inclusion $G_Q\subseteq\Stab_G(K)$ induces an injective homomorphism $$G_Q/H\hookrightarrow\Stab_G(K)/\langle\langle H\rangle\rangle;$$
\item all nontrivial elements in $\langle\langle H\rangle\rangle$ have translation length at least $L$ on $\calc_\Sigma(K)$.
\end{enumerate}
\end{lemma}

\begin{proof}
We focus on proving that $(H,Q)$ is a moving pair and satisfies the $(A,\epsilon)$-small cancellation condition with respect to the action of $\Stab_G(K)$ on $\calc_\Sigma(K)$; the rest of the conclusion then follows from Theorem~\ref{theo:dgo}. As in Definition~\ref{de:small-cancellation-subgroup}, we will denote by $\widehat{G}$ and $\widehat{H}$ the respective images of $\Stab_G(K)$ and $H$ in $\Map(\widehat{K})$, and by $\widehat{Q}$ the image of $Q$ under the natural identification between $\calc_\Sigma(K)$ and $\calc(\widehat{K})$.

We first prove that $(H,Q)$ is a moving pair for the $\Stab_G(K)$-action on $\calc_\Sigma(K)$. By the first assumption from Definition~\ref{de:small-cancellation-subgroup}, the pair $(\widehat{H},\widehat{Q})$ is a moving pair for the $\widehat{G}$-action on $\calc(\widehat{K})$, so $\widehat{Q}$ (and hence $Q$) is almost convex. Therefore, we only need to check that $H$ is normal in $G_Q$. By the second assumption of Definition~\ref{de:small-cancellation-subgroup}, the $\widehat{G}$-stabilizer of $\widehat{Q}$ has a lift $\widetilde{\Stab}(\widehat{Q})$ in $G_Q$ which contains $H$ and is made of elements that are supported on $K$. This implies that the extension $$1\to\Fix_G(K)\to G_Q\to\Stab_{\widehat{G}}(\widehat{Q})\to 1$$ is split and every element of $\widetilde{\Stab}(\widehat{Q})$ commutes with every element of $\Fix_G(K)$. Therefore $G_Q$ splits as a direct product isomorphic to $\widetilde{\Stab}(\widehat{Q})\times \Fix_G(K)$. In addition, as $\widehat{H}$ is normal in the $\widehat{G}$-stabilizer of $\widehat{Q}$, it follows that $H$ is normal in $\widetilde{\Stab}(\widehat{Q})$, whence in $G_Q$. 

We now prove that $(H,Q)$ satisfies the $(A,\epsilon)$-small cancellation condition with respect to the action of $\Stab_G(K)$ on $\calc_\Sigma(K)$. As $H\subseteq\widetilde{\Stab}(\widehat{Q})$, we have $H\cap\Fix_G(K)=\{1\}$, and therefore the injectivity radius of $(H,Q)$ is equal to the injectivity radius of $(\widehat{H},\widehat{Q})$. Let now $t\in\Stab_G(K)\setminus G_Q$. Then the image $\hat{t}$ of $t$ in $\widehat{G}$ does not stabilize $\widehat{Q}$, and $$\Delta_{\calc_{\Sigma}(K)}(Q,tQ)=\Delta_{\calc(\widehat{K})}(\widehat{Q},\hat{t}\widehat{Q}).$$ As $(\widehat{H},\widehat{Q})$ satisfies the $(A,\epsilon)$-small cancellation condition (Assumption~1 from Definition~\ref{de:small-cancellation-subgroup}), the conclusion follows. 
\end{proof}

\paragraph*{Step 2: Construction of small cancellation subgroups.}
This is the crucial place in the argument where we use the fact that $G$ is $K$-generic. 

\begin{lemma}\label{lemma:ppa}
Let $\Sigma$ be a connected orientable surface, let $K\subseteq\Sigma$ be a connected nondisplaceable subsurface of finite type, and let $G\subseteq\Map(\Sigma)$ be a $K$-generic subgroup. Let $g\in G$ be an element supported on $K$ which induces a symmetryless pseudo-Anosov mapping class of $\widehat{K}$. 

Then for every $A>0$ and $\epsilon>0$, there exists $n\in\mathbb{N}$ such that the cyclic group $\langle g^n\rangle$ is an $(A,\epsilon)$-small cancellation subgroup. 
\end{lemma}

\begin{proof}
Let $\widehat{g}\in \widehat{G}$ be the image of $g$. By \cite[Proposition~11]{BF}, the pseudo-Anosov element $\widehat{g}$ satisfies the WPD property for the action of $\widehat{G}$ on $\calc(\widehat{K})$. As $\widehat{g}$ is symmetryless, we deduce that there exist $n\in\mathbb{N}$ and a $\widehat{g}$-invariant almost convex subset $Q\subseteq\calc(\widehat{K})$ whose stabilizer in $\Map(\widehat{K})$ is equal to $\langle \widehat{g}\rangle$, such that $(\langle \widehat{g}^n\rangle ,Q)$ is a moving pair for the $\widehat{G}$-action on $\calc(\widehat{K})$ which satisfies the $(A,\epsilon)$-small cancellation condition (see e.g.\ \cite[Proposition~2.8]{Gui}). As $\langle g\rangle$ is a lift of $\langle\widehat{g}\rangle$ to $\Stab_G(K)$, the second assumption from Definition~\ref{de:small-cancellation-subgroup} is also satisfied. The lemma follows.
\end{proof}

\begin{lemma}\label{lemma:free-scc}
Let $\Sigma$ be a connected orientable surface, let $K\subseteq\Sigma$ be a connected nondisplaceable subsurface of finite type, and let $G\subseteq\Map(\Sigma)$ be a $K$-generic subgroup.

Then for every $A>0$ and every $\epsilon>0$, there exist a rank two nonabelian free subgroup $F\subseteq\Stab_G(K)$ made of elements supported on $K$, and a subspace $Q\subseteq\calc_\Sigma(K)$ stabilized by $F$, such that for every normal subgroup $N\unlhd F$, the pair $(N,Q)$ is an $(A,\epsilon)$-small cancellation pair.  
\end{lemma}

\begin{proof}
Let $g\in G$ be an element supported on $K$ which induces a symmetryless pseudo-Anosov mapping class of $\widehat{K}$ -- this exists as $G$ is $K$-generic. Let $L>0$, and let $A=A(L)>0$ and $\epsilon=\epsilon(L)>0$ be the constants provided by Lemma~\ref{lemma:small-cancellation-on-curve-graph}. By Lemma~\ref{lemma:ppa}, there exists $n\in\mathbb{N}$ such that $\langle g^n\rangle$ is an $(A,\epsilon)$-small cancellation subgroup. Applying Lemma~\ref{lemma:small-cancellation-on-curve-graph}, the normal subgroup $\langle\langle g^n\rangle\rangle$ of $\Stab_G(K)$ generated by $g^n$ is free, and every nontrivial element of $\langle\langle g^n\rangle\rangle$ induces a pseudo-Anosov mapping class of $\widehat{K}$ (because it acts on $\calc_\Sigma(K)$ with translation length at least $L$ by the third conclusion of Lemma~\ref{lemma:small-cancellation-on-curve-graph}). In addition every element of $\langle\langle g^n\rangle\rangle$ is supported on $K$, and $\langle\langle g^n\rangle\rangle$ is not cyclic because $G$ is $K$-nonelementary. Let $H=\langle\langle g^n\rangle\rangle$, and let $\widehat{H}$ be its image in $\Map(\widehat{K})$. Notice that the natural homomorphism $\Stab_G(K)\to\Map(\widehat{K})$ restricts to a bijection $H\to\widehat{H}$.

By \cite[Proposition~11]{BF}, every pseudo-Anosov element of $\Map(\widehat{K})$ has the WPD property for the $\Map(\widehat{K})$-action on $\calc(\widehat{K})$. As $\mathcal{E}_{\widehat{G}}(\widehat{g})$ is cyclic, we can apply Proposition~\ref{prop:small-cancellation-free-subgroup} and get a rank $2$ free subgroup $\widehat{F}\subseteq\widehat{H}$ and an almost convex subspace $\widehat{Q}\subseteq\calc(\widehat{K})$ such that $\widehat{F}=\Stab_{\widehat{G}}(\widehat{Q})$ and the moving pair $(\widehat{F},\widehat{Q})$ satisfies the $(A,\epsilon)$-small cancellation condition with respect to the $\widehat{G}$-action on $\calc(\widehat{K})$.

Let $F\subseteq H$ be the preimage of $\widehat{F}$ under the natural bijection $H\to\widehat{H}$, so $F$ is a rank two nonabelian free subgroup of $\Stab_G(K)$ made of elements supported on $K$. Let $Q\subseteq\calc_\Sigma(K)$ be the image of $\widehat{Q}$ under the natural identification between $\calc(\widehat{K})$ and $\calc_\Sigma(K)$. Then $Q$ is stabilized by $F$, and for every normal subgroup $N\unlhd F$, the pair $(N,Q)$ is an $(A,\epsilon)$-small cancellation pair, as required.
\end{proof}

\paragraph*{Step 3: Small cancellation on the quasi-tree of metric spaces.}

\begin{lemma}\label{lemma:small-cancellation-on-bbf}
Let $\Sigma$ be a connected orientable surface, let $K\subseteq\Sigma$ be a nondisplaceable subsurface of finite type, and let $G\subseteq\Map(\Sigma)$. There exists $L>0$ such that for every normal subgroup $N\unlhd\Stab_G(K)$, if every element of $N$ acts on $\calc(\widehat{K})$ with translation length at least $L$, then 
\begin{enumerate}
\item the normal subgroup $\langle\langle N\rangle\rangle$ of $G$ generated by $N$ is a free product of conjugates of $N$, and
\item the inclusion $\Stab_G(K)\subseteq G$ induces an injective homomorphism $$\Stab_G(K)/N\hookrightarrow G/\langle\langle N\rangle\rangle.$$ 
\end{enumerate}
\end{lemma}

\begin{proof}
Let $\mathbb{X}$ be the quasi-tree of metric spaces associated to the BBF family $\mathbb{Y}_K$ provided by Proposition~\ref{prop:bbf-family}. For every $K'\in\Map(\Sigma)\cdot K$, the space $\calc_\Sigma(K')$ embeds as a geodesically convex subspace in $\mathbb{X}$ (by the first conclusion of Theorem~\ref{theo:bbf}). Let $D>0$ be a constant (provided by the second conclusion of Theorem~\ref{theo:bbf}) such that for any two distinct $K_1,K_2\in\Map(\Sigma)\cdot K$, one has $\Delta_\mathbb{X}(\calc_\Sigma(K_1),\calc_\Sigma(K_2))<D$. Let $A>0$ and $\epsilon>0$ be constants given by the small cancellation theorem (Theorem~\ref{theo:dgo}), applied to the $G$-action on $\mathbb{X}$ with $C=0$.

As recalled above, the first assertion of Theorem~\ref{theo:bbf} ensures that $\calc_\Sigma(K)$ is geodesically convex in $\mathbb{X}$. It is also $\Stab_G(K)$-invariant, so we can (and shall) choose $L>0$ such that for every $g\in\Stab_G(K)$, if  $||g||_{\calc_\Sigma(K)}>L$, then $||g||_{\mathbb{X}}>\max\{A,D/\epsilon\}$.

We will now prove that $(N,\calc_\Sigma(K))$ satisfies the $(A,\epsilon)$-small cancellation condition with respect to the $G$-action on $\mathbb{X}$; the lemma will then follow from the geometric small cancellation theorem (Theorem~\ref{theo:dgo}). First, the group $N$ is normal in $\Stab_G(K)$, which is equal to the $G$-stabilizer of $\calc_\Sigma(K)$ in $\mathbb{X}$ (Lemma~\ref{lemma:fix-stab-curves}). Second, the injectivity radius of $N$ is at least equal to $\max\{A,D/\epsilon\}$. Finally, for every $t\in G\setminus\Stab_G(K)$, we have $$\Delta_{\mathbb{X}}(\calc_\Sigma(K),t\calc_\Sigma(K))< D<\epsilon~\inj_{\mathbb{X}}(N,\calc_\Sigma(K)).$$ The lemma follows.  
\end{proof}

\paragraph*{Conclusion.}

\begin{proof}[Proof of Theorem~\ref{theo:sq}]
In the whole proof, we let $L>0$ be the constant provided by Lemma~\ref{lemma:small-cancellation-on-bbf}, and we let $A=A(L)>0$ and $\epsilon=\epsilon(L)>0$ be the constants provided by Lemma~\ref{lemma:small-cancellation-on-curve-graph}.

We start by proving the existence of a normal nonabelian free subgroup of $G$. By Lemma~\ref{lemma:ppa}, there exists an element $k\in\Stab_G(K)$ supported on $K$ such that $\langle k\rangle$ is an $(A,\epsilon)$-small cancellation subgroup. Applying Lemma~\ref{lemma:small-cancellation-on-curve-graph}, we deduce that the normal subgroup $N_0$ of $\Stab_G(K)$ generated by $k$ is free and purely $K$-pseudo-Anosov, in fact all nontrivial elements of $N_0$ are supported on $K$ and act on $\calc_\Sigma(K)$ with translation length at least $L$. It is also nonabelian as $G$ is $K$-nonelementary. It therefore follows from Lemma~\ref{lemma:small-cancellation-on-curve-graph} that the normal subgroup of $G$ generated by $N_0$ is a nonabelian free group. 

We now turn to proving the second conclusion of Theorem~\ref{theo:sq}. By Lemma~\ref{lemma:free-scc}, we can find a rank two nonabelian free subgroup $F\subseteq\Stab_G(K)$ and a subspace $Q\subseteq\calc_\Sigma(K)$ stabilized by $F$ such that for every normal subgroup $N\unlhd F$, the pair $(N,Q)$ an $(A,\epsilon)$-small cancellation pair.

Let $\Gamma$ be a countable group. We aim to embed $\Gamma$ in a quotient $\overline{G}$ of $G$ in such a way that the image of $\Gamma$ in $\overline{G}$ is contained in the image of $F$ under the quotient map $G\twoheadrightarrow \overline{G}$.

As every countable group embeds in a two-generated group -- in other words $F_2$ is SQ-universal \cite[Theorem~10.3]{LS}, we can find a normal subgroup $N\unlhd F$ such that $\Gamma$ embeds in $F/N$. Our choice of $F$ and $Q$ ensures that $(N,Q)$ is an $(A,\epsilon)$-small cancellation pair. Let $N'$ be the normal closure of $N$ in $\Stab_G(K)$. Let $G_Q$ be the setwise stabilizer of $Q$ in $\Stab_G(K)$. Then $F/N$ embeds in $G_Q/N$, and by Lemma~\ref{lemma:small-cancellation-on-curve-graph} this in turns embeds in $\Stab_G(K)/N'$, and all elements of $N'$ act loxodromically on $\calc_\Sigma(K)$ with translation length at least $L$. Denoting by $N''$ the normal subgroup of $G$ generated by $N'$, Lemma~\ref{lemma:small-cancellation-on-bbf} then ensures that $\Stab_G(K)/N'$ embeds in $G/N''$. In conclusion we have proved that
\[
\Gamma \hookrightarrow F/N \hookrightarrow \Stab_G(K)/N' \hookrightarrow G/N'',
\]
concluding our proof.
\end{proof}

\subsection{Normal free subgroups and nondisplaceable finite-type subsurfaces}\label{sec:normal-free-subgroups}

In fact, nontrivial normal free subgroups in the mapping class group turns out to be a characterization of surfaces that contain nondisplaceable subsurfaces of finite type.

\begin{theo}
Let $\Sigma$ be a connected orientable surface of infinite type. The following statements are equivalent.
\begin{enumerate}
\item The surface $\Sigma$ contains a nondisplaceable subsurface of finite type.
\item The group $\Map(\Sigma)$ contains a nontrivial normal free subgroup.
\end{enumerate}
\end{theo}

\begin{proof}
The fact that $1\Rightarrow 2$ has been proved in Theorem~\ref{theo:sq}, so we focus on proving that $\neg 1\Rightarrow\neg 2$. Assume that every subsurface of $\Sigma$ of finite type is displaceable, and let $N$ be a normal subgroup of $\Map(\Sigma)$. We will prove that $N$ contains an abelian subgroup of rank $2$. 

Let $g\in N\setminus\{\mathrm{id}\}$. We first claim that there exists a Dehn twist $h\in\Map(\Sigma)$ that does not commute with $g$. Indeed otherwise $g$ would fix the isotopy class of every simple closed curve on $\Sigma$, and it follows that $g=\mathrm{id}$ by \cite{HHMV}. 

Let now $f=[g,h]$, which is nontrivial. Writing $f=(ghg^{-1})h^{-1}$, we see that $f$ is a product of two finitely supported mapping classes, so $f$ is supported on a subsurface $K$ of finite type. In addition, writing $f=g(hg^{-1}h^{-1})$, we see that $f\in N$. By assumption, there exists $\eta\in\Map(\Sigma)$ such that $\eta(K)\cap K\neq\emptyset$. Then the subgroup generated by $f$ and $\eta f\eta^{-1}$ is abelian of rank $2$, concluding our proof. 
\end{proof}

\section{Surfaces with no finite-type nondisplaceable subsurface}\label{sec:displaceable}

Under some light topological conditions on the surface $\Sigma$, we now aim to prove a converse statement to the work from Section~\ref{sec:nondisplaceable}, saying that in the absence of nondispaceable subsurfaces, the mapping class group $\Map(\Sigma)$ does not have any nonelementary continuous action on a hyperbolic space.

\subsection{Statement}\label{sec:statement}

We now present the two conditions we will impose on the surface $\Sigma$. 

\medskip

\noindent\textbf{Tameness of the end space.} The first one is a topological condition giving some control on the topology of the end space $E$ of $\Sigma$. The space $E$ is equipped with the following partial order \cite[Definition~4.1]{MR}: given $x,y\in E$, we let $x\preccurlyeq y$ if for every open neighborhood $U$ of $x$, there exist an open neighborhood $V$ of $y$ in $E$ and $f\in\Map(\Sigma)$ such that $f(V)\subseteq U$ and $f(V\cap E^g)\subseteq U\cap E^g$. We say that an end $x\in E$ is \emph{of maximal type} if it is maximal for the order $\preccurlyeq$.

Let $x\in E$. A neighborhood $U$ of $x$ is \emph{stable} \cite[Definition~4.14]{MR} if for every open neighborhood $U'\subseteq U$ of $x$, there is a homeomorphic copy of $(U,U\cap E^g)$ inside $(U',U'\cap E^g)$. Following \cite[Definition~6.14]{MR}, we say that $\Sigma$ has \emph{tame endspace} if every $x\in E$ which is either of maximal type, or an immediate predecessor of an end of maximal type (for the order $\preccurlyeq$) has a stable neighborhood. We refer to \cite[Section~6.3]{MR} for a thorough discussion of this condition. Also, we mention that even if we restrict to surfaces with tame endspace, we are still considering a large class, in particular there are uncountably many pairwise non-homeomorphic surfaces with tame endspace, see \cite[Remark~5.5]{FGM}. 

\medskip

\noindent \textbf{CB generation of the mapping class group.} The second condition we impose on $\Sigma$ roughly says that the mapping class group $\Map(\Sigma)$ has a well-controlled geometry, from the point of view of geometric group theory of Polish groups as developed by Rosendal in \cite{Ros}. 
Let $G$ be a Polish topological group. A subset $A\subseteq G$ is \emph{coarsely bounded}, abbreviated \emph{CB}, if for every continuous isometric $G$-action on a metric space $X$, the diameter of every $A$-orbit in $X$ is finite. The group $G$ is \emph{CB-generated} if it admits a generating subset which is coarsely bounded. Among surfaces with tame end spaces, surfaces whose mapping class group is CB-generated have been fully characterized by Mann and the third named author in \cite[Theorem~1.6]{MR}.

\medskip

The goal of the present section is to prove the following theorem.

\begin{theo}\label{theo:no-hyperbolic-action}
Let $\Sigma$ be a connected orientable surface with tame end space such that $\Map(\Sigma)$ is CB-generated. Assume that $\Sigma$ does not contain any nondisplaceable subsurface of finite type. 

Then $\Map(\Sigma)$ does not have any continuous nonelementary isometric action on a hyperbolic metric space.
\end{theo}

\begin{rk}
It would be interesting to know whether the continuity assumption can be removed from the statement of Theorem~\ref{theo:no-hyperbolic-action}, in other words, whether $\Map(\Sigma)$ admits any nonelementary isometric action on a hyperbolic metric space $X$ at all. A related question is whether $\Map(\Sigma)$ satisfies an automatic continuity property, saying that every homomorphism from $\Map(\Sigma)$ to a separable topological group (e.g.\ to $\Isom(X)$ where $X$ is a separable metric space) is continuous. This question has been recently answered in a positive way by Mann \cite{Man2} in some cases, when the end space of $\Sigma$ is the union of a Cantor set and a finite set. The question however seems to remain open in general.
\end{rk}

\subsection{An obstruction to continuous isometric actions on hyperbolic spaces}

Our proof of Theorem~\ref{theo:no-hyperbolic-action} relies on the following general criterion. We will check that $\Map(\Sigma)$ satisfies this criterion in the next section.

\begin{lemma}\label{lemma:no-hyperbolic-criterion}
Let $G$ be a topological group. Assume that there exists a split short exact sequence $$1\to N\to G\to A\to 1$$ with $A$ abelian and $N$ contained in the closure of a normal subgroup $H\unlhd G$ such that for every element $h\in H$, there exists a CB subgroup of $G$ that contains $h$. 

Then $G$ does not have any continuous nonelementary isometric action on a hyperbolic space.
\end{lemma}

\begin{proof}
Let $X$ be a hyperbolic space equipped with a continuous isometric action of $G$. We aim to prove that either all $G$-orbits in $X$ have finite diameter, or else that $G$ has a finite orbit in $\partial_\infty X$.
 
Since every element $h\in H$ is contained in some CB subgroup of $G$ (that might depend on $h$), it follows that every element of $H$ acts elliptically on $X$. Using Gromov's classification of isometric group actions on hyperbolic spaces (\cite{Gro}, see also e.g.\ \cite[Proposition~3.1]{CCMT}), we deduce that either all $H$-orbits in $X$ have finite diameter, or else the $H$-action on $X$ is horocyclic. 
In the latter case, as $H$ is normal in $G$, the unique point of $\partial_\infty X$ in the limit set of $H$ is $G$-invariant.

We can therefore assume that all $H$-orbits in $X$ have finite diameter. As the $G$-action on $X$ is continuous and $N$ is contained in the closure of $H$, it follows that all $N$-orbits in $X$ have finite diameter.  

Let $T$ be a lift of $A$ in $G$. Then every element of $G$ is a product of an element of $N$ and an element of $T$. As $T$ is abelian, either all $T$-orbits in $X$ have finite diameter, or else its limit set $\Lambda_\infty T\subseteq\partial_\infty X$ has cardinality $1$ or $2$. If all $T$-orbits in $X$ have finite diameter, then the same holds true for all $G$-orbits and we are done. 

We finally assume that $|\Lambda_\infty T|\in\{1,2\}$. Let $M\ge 0$ be sufficiently large so that $Y_M:=\{x\in X|\mathrm{diam}(N\cdot x)\le M\}$ is nonempty. By normality of $N$, the set $Y_M$ is $G$-invariant. In particular $Y_M$ is $T$-invariant, so denoting by $\Lambda_\infty Y_M\subseteq\partial_\infty X$ its limit set, we have $\Lambda_\infty T\subseteq\Lambda_\infty Y_M$. In addition, it follows from the definition of $Y_M$ that every point in $\Lambda_\infty Y_M$ is fixed by $N$. In particular $\Lambda_\infty T$ is $N$-invariant. As $G$ is generated by $N$ and $T$ it follows that $\Lambda_\infty T$ is a finite $G$-invariant set in $\partial_\infty X$. This concludes our proof.    
\end{proof}

\subsection{Avenue surfaces and proof of Theorem~\ref{theo:no-hyperbolic-action}}

\subsubsection{Avenue surfaces}

In proving Theorem~\ref{theo:no-hyperbolic-action}, we can always assume that $\Map(\Sigma)$ itself is not CB, as otherwise the conclusion is obvious (every continuous isometric action of $\Map(\Sigma)$ on any metric space has bounded orbits). For the remainder of this section, we will therefore assume that the surface $\Sigma$ is of the following form.

\begin{de}
An \emph{avenue surface} is a connected orientable surface $\Sigma$ which does not contain any nondisplaceable subsurface of finite type, whose end space is tame, and whose mapping class group $\Map(\Sigma)$ is CB-generated but not CB. 
\end{de}

The terminology \emph{avenue surface} comes from the fact that $\Sigma$ has exactly two maximal ends, as established in our next lemma. An example to keep in mind is the \emph{avenue of chimneys} from the introduction, depicted in Figure~\ref{fig:chimneys}. We recall (see Section~\ref{sec:statement}) that the endspace $E$ of $\Sigma$ is equipped with a partial order $\preccurlyeq$. This induces an equivalence relation on $E$, where two ends $x,y$ are equivalent if and only if $x\preccurlyeq y$ and $y\preccurlyeq x$. In the sequel, when we talk about equivalence classes of ends, this will always be with respect to this equivalence relation. 

\begin{lemma}\label{lemma:avenue}
Let $\Sigma$ be an avenue surface. Then $\Sigma$ has either $0$ or infinite genus, and $\Sigma$ has exactly two ends of maximal type. 
\end{lemma}

\begin{figure}
\centering
\def\JPicScale{.8}
\input{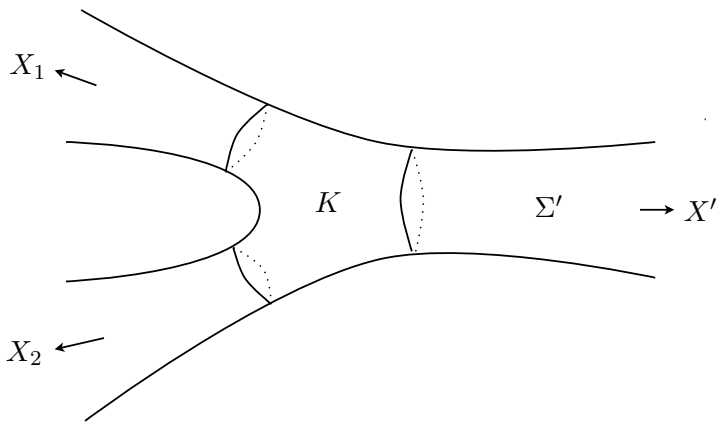}
\caption{The situation in the third paragraph of the proof of Lemma~\ref{lemma:avenue}.}
\label{fig:nondisplaceable}
\end{figure}

\begin{proof}
Notice first as in \cite[Example~2.4]{MR} that the genus of $\Sigma$ is either $0$ or infinite, as otherwise any finite type subsurface of $\Sigma$ whose genus matches that of $\Sigma$ would be nondisplaceable. 

Let $E$ be the end space of $\Sigma$, and let $\mathcal{M}$ be the set of all maximal elements in $E$. We aim to prove that $|\mathcal{M}|=2$. By \cite[Proposition~4.7]{MR}, the set $\mathcal{M}$ is nonempty. Also $|\mathcal{M}|\neq 1$: otherwise \cite[Lemma~4.12]{MR} implies that $(E,E^g)$ is \emph{self-similar} in the sense of \cite[Section~3.1]{MR}.\footnote{This means that for any partition $E=E_1\dunion\cdots\dunion E_n$ into clopen subsets, there exist $i\in\{1,\dots,n\}$ and a clopen subset $D\subseteq E_i$ such that $(D,D\cap E^g)$ is homeomorphic to $(E,E^g)$.} As $\Sigma$ has $0$ or infinite genus, \cite[Proposition~3.1]{MR} then implies that $\Map(\Sigma)$ is globally CB, contradicting our assumption that $\Sigma$ is an avenue surface.

We first claim that $\mathcal{M}$ does not contain any infinite equivalence class. Indeed, otherwise, let $X$ be such an equivalence class. By \cite[Proposition~4.7]{MR}, the class $X$, viewed as a subspace of $E$, is a Cantor set. If $X$ is the unique maximal equivalence class, then \cite[Lemma~4.13]{MR} implies that $(E,E^g)$ is self-similar, and therefore $\Map(\Sigma)$ is globally CB as above, a contradiction. So assume that there exists another maximal equivalence class $X'$ (that may be finite or infinite). We can partition $X$ into two disjoint closed subsets $X_1,X_2$. As $X'$ is a closed subset of $E$ (see \cite[Lemma~4.6]{MR}), it follows from Lemma~\ref{lemma:partition-ends} that there exists a finite-type subsurface $K$ of $\Sigma$ that pairwise separates $X_1,X_2$ and $X'$ -- see Figure~\ref{fig:nondisplaceable}. We now claim that this subsurface $K$ is nondisplaceable, which gives a contradiction.

To prove our claim that $K$ is nondisplaceable, let $\Sigma'$ be the connected component of $\Sigma\setminus K$ whose end set contains $X'$, and assume towards a contradiction that there exists $f\in\Homeo(\Sigma)$ such that $f(K)\cap K=\emptyset$. Then $f(K)$ cannot be contained in a component of $\Sigma\setminus K$ distinct from $\Sigma'$, as otherwise one of the complementary components of $f(\Sigma)$ would contain ends in both $X$ and $X'$. And $f(K)$ cannot be contained in $\Sigma'$ either as otherwise only one complementary components of $f(\Sigma)$ would contain ends in $X$. This contradiction shows that $K$ is nondisplaceable.  

Therefore, every equivalence class of maximal elements is finite, and by \cite[Lemma~5.3]{MR}, there are finitely many such classes. So $\mathcal{M}$ is finite. There remains to prove that $|\mathcal{M}|\le 2$, so assume towards a contradiction that $|\mathcal{M}|\ge 3$. By Lemma~\ref{lemma:partition-ends}, there exists a finite-type subsurface $K$ of $\Sigma$ that pairwise separates all ends in $\mathcal{M}$. We claim that the surface $K$ is nondisplaceable. Indeed, otherwise, there would exist $f\in\Homeo(\Sigma)$ such that $K$ is contained in one complementary component $\Sigma'$ of $f(K)$. Thus $\Ends(\Sigma')$ would contain two ends in $\calm$, which is impossible as $f^{-1}(\Sigma')$ is a complementary component of $K$. We have thus proved that $K$ is nondisplaceable: this contradiction completes our proof.
\end{proof}

\begin{lemma}\label{lemma:E(A,B)}
Let $\Sigma$ be an avenue surface, let $E$ be its endspace, and let $x_A,x_B$ be the two maximal ends of $\Sigma$. Then for every $x\in E\setminus\{x_A,x_B\}$, the equivalence class of $x$ accumulates to both $x_A$ and $x_B$.
\end{lemma}

\begin{proof}
Let $[x]$ be the equivalence class of $x$. Assume towards a contradiction that $[x]$ does not accumulate to $x_B$. The set $\overline{[x]}\cup\{x_A\}$ is a closed subset of $E$ of cardinality at least two, so it admits a nontrivial partition $\overline{[x]}\cup\{x_A\}=X_1\dunion X_2$ into two subsets which are both closed in $E$. By Lemma~\ref{lemma:partition-ends}, there exists a finite-type subsurface $K\subseteq\Sigma$ that pairwise separates the closed sets $X_1,X_2$ and $\{x_B\}$. We claim that $K$ is nondisplaceable: this will give a contradiction, showing that $[x]$ accumulates to $x_B$. By symmetry $[x]$ also accumulates to $x_A$, which completes our proof. 

We are thus left proving the above claim, that $K$ is nondisplaceable. Assume towards a contradiction that there exists $f\in\Homeo(\Sigma)$ such that $f(K)\cap K\neq\emptyset$. Let $\Sigma_B$ be the complementary component of $\Sigma\setminus K$ that contains $x_B$. Then $f(K)$ cannot be contained in a complementary component of $K$ distinct from $\Sigma_B$, as otherwise one complementary component of $f(K)$ would contain both $x_B$ and ends in $[x]$ -- while no complementary component of $K$ does. And $f(\Sigma)$ cannot be contained in $\Sigma_B$, as otherwise one of the complementary components of $f(K)$ would contain both $x_A$ and all ends in $[x]$ -- while no complementary component of $K$ does. This contradiction shows that $K$ is nondisplaceable.
\end{proof}

\subsubsection{Horizontally bounded mapping classes}

\begin{de} 
Let $\Sigma$ be an avenue surface, and let $x_A,x_B$ be the two maximal ends of $\Sigma$. 

A subsurface $R$ of $\Sigma$ (possibly of infinite type) is \emph{horizontally bounded} if $R$ is disjoint from some neighborhood of $x_A$ and from some neighborhood of $x_B$, i.e.\ if $\Ends(R)\cap\{x_A,x_B\}\neq\emptyset$. 

An element $f\in\Map(\Sigma)$ is \emph{horizontally bounded} if $f$ has a representative in $\Homeo(\Sigma)$ which is supported on a horizontally bounded subsurface. 
\end{de} 

An example of a horizontally bounded subsurface is the subsurface $R$ represented in Figure~\ref{fig:chimneys} on the avenue of chimneys. We will also say that a horizontally bounded subsurface is \emph{standard} if it is bounded by exactly two separating curves. Notice that every horizontally bounded subsurface is contained in a standard one (because every neighborhood of either $x_A$ or $x_B$ in $\Sigma\cup E(\Sigma)$ contains a subneighborhood bounded by a single separating curve, as follows from Lemma~\ref{lemma:partition-ends}). Therefore, every horizontally bounded mapping class has a representative in $\Homeo(\Sigma)$ that is supported on a standard horizontally bounded subsurface.

We let $\Map^0(\Sigma)$ be the subgroup of $\Map(\Sigma)$ of index at most $2$ made of all mapping classes that fix the two maximal ends $x_A$ and $x_B$, as opposed to permuting them. We denote by $\Bdd(\Sigma)$ the subset of $\Map^0(\Sigma)$ made of all horizontally bounded mapping classes. 

\begin{lemma}\label{lemma:bdd-normal}
Let $\Sigma$ be an avenue surface. The subset $\Bdd(\Sigma)$ is a normal subgroup of $\Map^0(\Sigma)$. 
\end{lemma}

\begin{proof}
That $\Bdd(\Sigma)$ is a subgroup follows from the observation that the union of two horizontally bounded subsurfaces is again horizontally bounded. Normality follows from the fact that if $f\in\Bdd(\Sigma)$ and $g\in\Map^0(\Sigma)$, then the support of $gfg^{-1}$ is equal to the $g$-translate of the support of $f$, again a horizontally bounded subsurface. 
\end{proof}

\subsubsection{Horizontally bounded mapping classes are contained in CB subgroups}

\begin{lemma}\label{lemma:homeomorphic-subsurfaces}
Let $\Sigma$ be an avenue surface, with maximal ends $x_A$ and $x_B$. Let $R_1,R_2\subseteq\Sigma$ be two standard horizontally bounded subsurfaces. Assume that 
\begin{enumerate}
\item $R_1$ and $R_2$ have the same genus (possibly infinite),
\item for every maximal countable equivalence class $C$ of ends in $E\setminus\{x_A,x_B\}$, the intersections $C\cap\Ends(R_1)$ and $C\cap\Ends(R_2)$ have the same cardinality, and
\item $R_1$ and $R_2$ both contain a representative of every uncountable equivalence class of ends. 
\end{enumerate}
Then $K_1$ and $K_2$ are homeomorphic.
\end{lemma}

\begin{proof}
The proof is analogous to the argument found in \cite[Lemma~6.17]{MR}. For every $i\in\{1,2\}$, let $X_i$ be the set of ends that are maximal in $E\setminus\{x_A,x_B\}$ and contained in $\Ends(R_i)$, and whose equivalence class in $E$ is countable. By \cite[Observation~6.12]{MR}, the sets $X_1$ and $X_2$ are finite, and our second assumption ensures that there is a bijection $\theta:X_1\to X_2$ such that for every $y\in X_1$, the ends $y$ and $\theta(y)$ are equivalent. As $\Sigma$ has tame endspace, for every $i\in\{1,2\}$, every end $y\in X_i$ has a stable neighborhood $V_{i,y}$. As any two stable neighborhoods of equivalent ends are homeomorphic \cite[Lemma~4.17]{MR}, for every $y\in X_1$, the neighborhoods $V_{1,y}$ and $V_{2,\theta(y)}$ are homeomorphic.

For every $i\in\{1,2\}$, let $$W_i=\Ends(R_i) - \bigcup_{y\in X_i} V_{i,y}.$$ We claim that for every $i\in\{1,2\}$, the space $W_i\cup\Ends(R_{3-i})$ is homeomorphic to $\Ends(R_{3-i})$ by a homeomorphism preserving the subspace made of ends that are accumulated by genus. Indeed, using \cite[Proposition~4.7]{MR}, we see that for every point $w\in W_i$, there exists a point $w'\in\Ends(R_{3-i})$ of maximal type in $E\setminus\{x_A,x_B\}$ such that $w'$ is an accumulation point of $E(w)$. As $\Sigma$ has tame endspace, the point $w'$ has a stable neighborhood. By \cite[Lemma~4.18]{MR}, there exist a clopen neighborhood $U_w$ of $x$ and a stable neighborhood $V_{w'}$ of $w'$ such that $U_w\cup V_{w'}$ is homeomorphic to $V_{w'}$. As $W_i$ is compact, it is covered by finitely many neighborhoods $U_w$. By considering all intersections of these neighborhoods, we can in fact write $W_i$ as the disjoint union of finitely many neighborhoods $U_w$ with the above property. The claim follows.

The above claim implies that $\Ends(R_1)$ is homeomorphic to $$W_2\cup\Ends(R_1)=W_1\cup \left(W_2\cup\bigcup_{y\in X_1} V_{1,y}\right),$$ which in turn is homeomorphic to $$W_1\cup\Ends(R_2)=W_1\cup \left(W_2\cup\bigcup_{y\in X_1}V_{2,\theta(y)}\right),$$ and finally to $\Ends(R_2)$. All these homeomorphisms preserve the subspaces made of ends accumulated by genus. So $R_1$ and $R_2$ have homeomorphic endspaces. As in addition they have the same genus, and they both have two boundary curves (because they are standard), the conclusion follows from Richards's classification of infinite-type surfaces \cite{Ric}.
\end{proof}

\begin{lemma}\label{lemma:hb-in-avenue}
Let $\Sigma$ be an avenue surface, let $x_A,x_B$ be the two maximal ends of $\Sigma$, and let $R$ be a standard horizontally bounded subsurface of $\Sigma$.

Then for every neighborhood $U$ of either $x_A$ or $x_B$ in $\Sigma$, there exists a homeomorphism $\eta$ of $\Sigma$ such that $\eta(R)\subseteq U$.
\end{lemma}

\begin{proof}
Let $E$ be the endspace of $\Sigma$. By Lemma~\ref{lemma:E(A,B)}, every equivalence class of ends in $E\setminus\{x_A,x_B\}$ accumulates to both $x_A$ and $x_B$. In addition, as $\Map(\Sigma)$ is CB-generated, it follows from \cite[Lemma~6.4]{MR} that no end set of $\Sigma$ has limit type in the sense of \cite[Definition~6.2]{MR}. Up to increasing $R$, we can therefore assume that $R$ contains an end from every equivalence class in $E\setminus\{x_A,x_B\}$.

We denote by $E_{\rm mc}(x_A,x_B)$ the subspace of $E\setminus\{x_A,x_B\}$ made of all ends that are maximal in $E\setminus\{x_A,x_B\}$ and whose equivalence class in $E$ is countable. By \cite[Lemma~6.13]{MR}, the set $E_{\rm mc}(x_A,x_B)$ is a union of finitely many equivalence classes of ends. Without loss of generality, we will assume that the neighborhood $U$ is standard, i.e.\ bounded by a single separating curve $\alpha$ on $\Sigma$. As $U$ contains a representative of every equivalence class in $E\setminus\{x_A,x_B\}$ (Lemma~\ref{lemma:E(A,B)}) and the endspace of $\Sigma$ is not of limit type, we can find a separating curve $\beta$ so that denoting by $R'$ the subsurface bounded by $\alpha$ and $\beta$, the set $\Ends(R')$ contains a representative of every equivalence class of ends in $E\setminus\{x_A,x_B\}$, and contains at least as many representatives from every class of $E_{\rm mc}(x_A,x_B)$ as $R$. Up to removing a stable neighborhood of some of the equivalence classes in $E_{\rm mc}(x_A,x_B)$, we can assume that $\Ends(R)$ and $\Ends(R')$ have the same number of representatives in each of these classes. By adjusting the genus, we can arrange that $R$ and $R'$ are homeomorphic, and the lemma follows.   
\end{proof}

\begin{lemma}\label{lemma:bdd-cb}
Let $\Sigma$ be an avenue surface. Every element in $\Bdd(\Sigma)$ is contained in a CB subgroup of $\Map(\Sigma)$.
\end{lemma}

\begin{proof}
By \cite[Theorem~5.7]{MR}, there exists a finite-type subsurface $K\subseteq\Sigma$ such that the subgroup $H\subseteq\Map(\Sigma)$ made of all mapping classes which are represented by a homeomorphism of $\Sigma$ supported on $\Sigma\setminus K$ is CB in $\Map(\Sigma)$. Now let $f\in\Bdd(\Sigma)$. By definition, the mapping class $f$ has a representative supported on a standard horizontally bounded subsurface $R$. Lemma~\ref{lemma:hb-in-avenue} therefore ensures that that there exists a homeomorphism $\eta$ of $\Sigma$ such that $\eta(R)\subseteq \Sigma\setminus K$. It follows that $\eta f\eta^{-1}\in H$. Therefore $f$ is contained in $\eta^{-1}H\eta$, which is a CB subgroup of $\Map(\Sigma)$.
\end{proof}

\subsubsection{The kernel of the twist homomorphisms}

\begin{lemma}\label{lemma:twist-morphism}
Let $\Sigma$ be an avenue surface. There exist $n\in\mathbb{N}$ and a split exact sequence 
\[
1\to N\to \Map^0(\Sigma) \to\mathbb{Z}^n\to 1
\] 
such that $N$ is contained in the closure of $\Bdd(\Sigma)$. 
\end{lemma}

\begin{proof}
Let $E$ be the end space of $\Sigma$. By \cite[Proposition~5.4]{MR}, the space $E$ has a partition $E=A\dunion B$ into two disjoint self-similar subsets, each of which contains exactly one of the two ends of $\Sigma$ of maximal type $x_A,x_B$. Following \cite[Definition~6.11]{MR}, we let $E_{\rm mc}(x_A,x_B)$ be the subspace of $E\setminus\{x_A,x_B\}$ made of all ends that are maximal in $E\setminus\{x_A,x_B\}$ and whose equivalence class in $E$ is countable.

We aim to define a homomorphism $\Phi:\Map^0(\Sigma)\to\mathbb{Z}^n$ for some $n\in\mathbb{N}$. Let $k$ be the number of equivalence classes of ends contained in $E_{\rm mc}(x_A,x_B)$, which is finite by \cite[Lemma~6.13]{MR}, and let $[y_1],\dots,[y_k]$ be these equivalence classes. In other words 
$E_{\rm mc}(x_A,x_B) = \cup_{i=1}^k [y_i]$. Consider the map 
\[
\Psi \from \Map^0(\Sigma) \to \mathbb{Z}^k, \qquad \Psi(f) = (n_1(f), \dots, n_k(f)),
\]
where for every $i\in\{ 1, \dots, k\}$, we let
\[
n_i(f) = \left| \Big\{ x \in [y_i] \, \Big| \, x \in A, \quad f(x) \in B\Big\}\right|
- \left| \Big\{ x \in [y_i] \, \Big| \, x \in B, \quad f(x) \in A\Big\}\right|. 
\]
Notice that this is well-defined: indeed, every mapping class $f\in\Map^0(\Sigma)$ fixes $x_A$ and $x_B$, and as the equivalences classes $[y_i]$ are countable -- whence discrete in $E$, see \cite[Observation~6.12]{MR} -- it follows that the two quantities that appear in the definition of $n_i(f)$ are finite. In addition $\Psi$ is a homomorphism: indeed, one checks that $n_i(g\circ f)=n_i(g)+n_i(f)$ by partitioning the points $x\in E$ into eight subsets, depending whether $x$ belongs to $A$ or $B$, whether $f(x)$ belongs to $A$ or $B$, and whether $g\circ f(x)$ belongs to $A$ or $B$, counting the points in each of these subsets of recording their contribution to $n_i$.

If some end in $E\setminus\{x_A,x_B\}$ is accumulated by genus, or if $\Sigma$ has genus $0$, then we let $n=k$ and $\Phi=\Psi$. 

If $\Sigma$ has infinite genus and no end in $E\setminus\{x_A,x_B\}$ is accumulated by genus, then we let $n=k+1$. In this case, following work of Aramayona, Patel and Vlamis \cite[Section~3]{APV}, we define an extra homomorphism $\Phi_g:\Map^0(\Sigma)\to\mathbb{Z}$ in the following way (and we let $\Phi=(\Psi,\Phi_g)$). Let $c$ be a separating curve on $\Sigma$ which separates the two ends $x_A,x_B$, and let $f\in\Map^0(\Sigma)$. Let $R\subseteq\Sigma$ be a horizontally bounded subsurface of $\Sigma$ which is bounded by two separating curves $\alpha_A,\alpha_B$ which both separate $x_A$ from $x_B$, such that both $c$ and $f(c)$ are contained in $R$ -- with $\alpha_A$ closer to $x_A$ than $\alpha_B$. Then $c$ separates $R$ into two subsurfaces $R_A$ (containing $\alpha_A$ in its boundary) and $R_B$ (containing $\alpha_B$ in its boundary), while $f(c)$ separates $R$ into two subsurfaces $R'_A$ (containing $\alpha_A$ in its boundary) and $R'_B$ (containing $\alpha_B$ in its boundary). We then let $\Phi_g(f)=\genus(R'_A)-\genus(R_A)$. This quantity does not depend on the choice of $R$, and $\Phi_g$ is the desired homomorphism, see \cite[Proposition~3.3]{APV}.

We claim that in all cases, the homomorphism $\Phi:\Map^0(\Sigma)\to\mathbb{Z}^n$ is surjective and admits a section. Indeed, a section of $\Psi$ was constructed in Step~3 of the proof of \cite[Proposition~6.18]{MR}. When $\Sigma$ has infinite genus and no end in $E\setminus\{x_A,x_B\}$ is accumulated by genus, a section of $\Phi_g$ is given by the cyclic subgroup generated by a handle shift as defined in \cite[Section~6]{PV} (see also \cite[Definition~6.20]{MR}).

Finally, let $N\unlhd\Map^0(\Sigma)$ be the kernel of $\Phi$, and let $f \in N$; we aim to prove that $f$ belongs to the closure of $\Bdd(\Sigma)$. Using Lemma~\ref{lemma:hb-in-avenue} (applied to a horizontally bounded subsurface $R$ such that $\Ends(R)$ contains a representative of every equivalence class of ends in $E\setminus\{x_A,x_B\}$), we can find an exhaustion of $\Sigma$ by subsurfaces $R_j$ (with $j\in\mathbb{N}$) that are all horizontally bounded, so that the boundary of each $R_j$ consists of two separating curves, and for every $j\in\mathbb{N}$, each of the two connected components of $R_j\setminus R_{j-1}$ contains a representative of every uncountable equivalence class in $E\setminus\{x_A,x_B\}$. We can assume that for every $j\in\mathbb{N}$, one of the complementary components of $R_j$ has all its ends contained in $A$ and the other has all its ends contained in $B$, and the same holds true for $f(R_j)$.

Let $\ell \from {\mathbb N} \to {\mathbb N}$ be a map such that 
for every $j\in\mathbb{N}$ we have
\[
f(R_j) \subseteq R_{\ell(j)-1}. 
\]
Then $(R_{\ell(j)} - R_j)$ and $(R_{\ell(j)} - f(R_j))$ both have two connected components. 
We denote the components of $(R_{\ell(j)} - R_j)$ by $Q_j^-$ and $Q_j^+$
and the components of $(R_{\ell(j)} - f(R_j))$ by $T_j^-$ and $T_j^+$. 
Since $f \in N$, for every $i\in\{1, \dots, k\}$, we have 
$$|Q_j^-\cap [y_i]|=|T_j^-\cap [y_i]|\qquad\text{and}\qquad |Q_j^+\cap [y_i]|=|T_j^+\cap [y_i]|,$$
\[
\genus(Q_j^-) = \genus (T_j^-)
\qquad\text{and}\qquad
\genus(Q_j^+) = \genus (T_j^+). 
\]
Furthermore, the choice of our exhaustion ensures that $Q_j^+,Q_j^-,T_j^+,T_j^-$ all contain a representative of every uncountable equivalence class in $E$ (notice here that it was important to assume that $f(R_j)$ is contained in $R_{\ell(j)-1}$ and not only in $R_{\ell(j)}$). Therefore, Lemma~\ref{lemma:homeomorphic-subsurfaces} ensures that
$Q_j^-$ is homeomorphic to $T_j^-$ and $Q_j^+$ is homeomorphic to $T_j^+$. Therefore, there is a homeomorphism $f_j \in \Map(R_{\ell(j)})$ such that 
$f_j \big|_{R_j} = f\big|_{R_j}$. As $f_j \in \Bdd(\Sigma)$, it follows that $f$ is in the closure 
of $\Bdd(\Sigma)$, as desired. 
\end{proof}

We are now in position to conclude our proof of the main theorem of the section.

\begin{proof}[Proof of Theorem~\ref{theo:no-hyperbolic-action}]
We can assume that $\Map(\Sigma)$ is not CB, as otherwise the conclusion is obvious. Therefore $\Sigma$ is an avenue surface. In view of Lemmas~\ref{lemma:bdd-normal},~\ref{lemma:bdd-cb} and~\ref{lemma:twist-morphism}, the group $G=\Map^0(\Sigma)$ satisfies all assumptions from Lemma~\ref{lemma:no-hyperbolic-criterion}, with $H=\Bdd(\Sigma)$. Therefore every continuous isometric action of $\Map^0(\Sigma)$ on a hyperbolic space is elementary. As $\Map^0(\Sigma)$ has finite index in $\Map(\Sigma)$, the same conclusion holds true for $\Map(\Sigma)$.
\end{proof}

\footnotesize

\bibliographystyle{alpha}
\bibliography{bbf}

\begin{flushleft}
Camille Horbez\\
CNRS\\ 
Universit\'e Paris-Saclay, CNRS,  Laboratoire de math\'ematiques d'Orsay, 91405, Orsay, France \\
\emph{e-mail:}\texttt{camille.horbez@universite-paris-saclay.fr}\\
\end{flushleft}

\begin{flushleft}
Yulan Qing\\
Department of Mathematics, University of Toronto, Toronto, ON, Canada\\
\emph{e-mail:}\texttt{yqing@math.toronto.edu} 
\end{flushleft}

\begin{flushleft}
Kasra Rafi\\
Department of Mathematics, University of Toronto, Toronto, ON, Canada\\
\emph{e-mail:}\texttt{rafi@math.toronto.edu} 
\end{flushleft}

\end{document}